\documentclass[11pt, twoside]{article}

\usepackage[english]{babel}

\usepackage{amssymb}
\usepackage{mathrsfs}
\usepackage{amsmath}
\usepackage{amsthm}
\usepackage{amsfonts}
\usepackage{latexsym}
\usepackage{indentfirst}
\usepackage{color}
\usepackage{txfonts}
\usepackage{enumerate}

\usepackage[colorlinks=true,
linkcolor=blue,
citecolor=red,
urlcolor=magenta,
]{hyperref}

\usepackage{txfonts}
\usepackage{anysize}

\allowdisplaybreaks

\newtheorem{theorem}{Theorem}[section]
\newtheorem{lemma}[theorem]{Lemma}

\newtheorem{proposition}[theorem]{Proposition}

\theoremstyle{definition}
\newtheorem{remark}[theorem]{Remark}

\newcounter{assum}

\renewcommand{\appendix}{\par
\setcounter{section}{0}%
\setcounter{subsection}{0}%
\setcounter{subsubsection}{0}%
\gdef\thesection{\@Alph\c@section}%
\gdef\thesubsection{\@Alph\c@section.\@arabic\c@subsection}%
\gdef\theHsection{\@Alph\c@section.}%
\gdef\theHsubsection{\@Alph\c@section.\@arabic\c@subsection}%
\csname appendixmore\endcsname
}

\numberwithin{equation}{section}

\begin{document}
\title{\bf\Large
A Sharp Planar Fractional Isoperimetric Inequality\footnotetext{\hspace{-0.35cm} 2020 {\it
Mathematics Subject Classification}.
Primary 49Q20; Secondary 46E35, 52A40, 52A10, 28A75.
\endgraf {\it Key words and phrases.}
fractional perimeter, sharp isoperimetric inequality, convex body, chord functional.
\endgraf This project is partially supported by the National Natural
Science Foundation of China (Grant Nos. 12431006, 12371093, 12501118, and 12601167),
the Beijing Natural Science Foundation (Grant No. 1262011), the
Natural Science Foundation of Fujian Province (Grant No. 2026J008197),
the Fundamental Research Funds for the Central Universities
(Grant No. 2253200028), and Longyuan Young Talents of Gansu Province.
}}
\author{Xiaosheng Lin, Dachun Yang,
Sibei Yang, Wen Yuan and Yangyang Zhang
}
\maketitle

\vspace{-0.7cm}

\begin{center}
\begin{minipage}{13cm}
{\small {\bf Abstract.}\quad
In this article, we prove the sharp form of the
fractional isoperimetric inequality in the plane,
originally posed by Maz'ya [Problem 1, Integral
Equations Operator Theory, 2018]. More precisely,
we show that, for any given $s\in(0,1)$ and
any bounded domain $\Omega \subset \mathbb{R}^2$
with $C^1$ boundary,
$$
P_s(\Omega) \le \frac{\pi^{s-\frac{1}{2}}\Gamma\left(\frac{3-s}{2}\right)}
  {s(1-s)\Gamma\left(\frac{4-s}{2}\right)}
  \left[\mathcal{H}^1(\partial\Omega)\right]^{2-s},
$$
where $P_s$ denotes the fractional $s$-perimeter,
$\Gamma$ denotes the Gamma function,
and $\mathcal{H}^{1}$ denotes the $1$-dimensional
Hausdorff measure on $\mathbb{R}^2$. The constant is sharp
and equality is attained by the disk.
The proof proceeds in three steps: reducing the fractional perimeter
to a chord functional, reducing connected domains to convex bodies,
and proving the sharp convex chord inequality
via a fractional Willmore-type inequality, a variational
formula along outer parallel bodies, and asymptotic analysis.}
\end{minipage}
\end{center}

\vspace{0.1cm}

\tableofcontents
\section{Introduction}

The classical isoperimetric inequality is one of the most fundamental tools of modern mathematics,
asserting that among all sets of prescribed volume in Euclidean space, balls have the least perimeter
(see, for example, \cite{c01,d58,f15,f04,o78}).
In its modern formulation due to De Giorgi \cite{d58}, it states that, for any Borel set $E \subset \mathbb{R}^n$
with finite Lebesgue measure $|E|$,
$$
P(E) \ge n\omega_n^{\frac1n}|E|^{\frac{n-1}n},
$$
with equality holding if and only if $E$ is a ball. Here, $n\ge2$, $P(E)$ denotes the \emph{distributional perimeter}
of $E$ (which coincides with the classical $(n-1)$-dimensional measure of $\partial E$ when $E$ has a smooth boundary
$\partial E$), and  $\omega_n$ denotes the volume of the unit ball in $\mathbb{R}^n$.
Various quantitative versions of the isoperimetric inequality have garnered considerable interest in the past decades
(see, for example, \cite{cgprs22,fmp10,fmp08,hlpry25,hl25,lxz11,x20,x12,x09,x07}),
thanks to its wide-ranging applications across geometry, linear and nonlinear
partial differential equations, and probability theory.

More recently, attention has turned to nonlocal analogues of the isoperimetric inequality
(see, for example, \cite{adm11, crs10,cmr26,dnrv15,ffmmm15,fls08,fs08,fmm11,l14,nps18}).
In particular, for any given $s \in (0,1)$, the \emph{fractional $s$-perimeter} of a Borel
set $E \subset \mathbb{R}^n$ is defined by setting
$$
P_s(E) := \int_E \int_{\mathbb{R}^n \setminus E} \frac{dx\,dy}{|x-y|^{n+s}}.
$$
Frank and Seiringer  \cite{fs08} established the fractional isoperimetric inequality that,
for any Borel set $E \subset \mathbb{R}^n$ with finite Lebesgue measure $|E|$,
$$
|E|^{\frac{n-s}n}\le C_{n,s}P_s(E)
$$
for a suitable positive constant $C_{n,s}$ depending only on both $n$ and $s$, with equality
holding if and only if $E$ is a ball. Quantitative versions of this inequality have been
established by Fusco et al. \cite{fmm11} and by Figalli et al. \cite{ffmmm15}. Moreover,
a strong form of the fractional quantitative isoperimetric inequality has recently been
established by Cinti et al. \cite{cmr26}.

Despite these advances, the problem of determining the sharp constant in the fractional isoperimetric inequality,
in the sense of the sharp inequality proposed by Maz'ya \cite[Problem 1, p.\,5]{m18}, has
remained open. Specifically, let $n\ge2$, $s\in(0,1)$, and $\Omega\subset\mathbb{R}^n$ be a bounded
domain with smooth boundary $\partial\Omega$.
Maz'ya \cite[Problem 1, p.\,5]{m18} asks for the sharp constant $C$ in the isoperimetric inequality
\begin{equation}\label{e1.1}
P_s(\Omega)\le C[\mathcal{H}^{n-1}(\partial\Omega)]^{\frac{n-s}{n-1}},
\end{equation}
where $\mathcal{H}^{n-1}$ denotes the $(n-1)$-dimensional Hausdorff measure on $\mathbb{R}^n$.
We also mention that the sharp constant $C$ in the isoperimetric inequality \eqref{e1.1} is the
same as the sharp constant $C$ in the fractional Sobolev inequality
\begin{equation}\label{e1.1x}
\left[\int_{\mathbb{R}^{n}}\int_{\mathbb{R}^{n}}\frac{|u(x)-u(y)|^{q}}
{|x-y|^{n+s}}\,dx\,dy\right]^{\frac1q}
\le (2C)^{\frac1q}  \| \nabla u \|_{L^{1}(\mathbb{R}^{n})},
\end{equation}
where $q:= \frac{n-s}{n-1}$ and $u \in C_{\rm c}^{\infty}(\mathbb{R}^{n})$
(the set of all infinitely differentiable functions on $\mathbb{R}^n$
with compact support)
(see, for example, \cite[Section 6]{m03}). Meanwhile, by \cite[p.\,333]{m03},
we find that the sharp constant $C$ in \eqref{e1.1} have an upper bound
$$
C \le \frac{n \omega_{n}^{(n-1+s)/n}}{s(1-s)}
\left[\mathcal{H}^{n-1}(\partial B_{1})\right]^{\frac{n-s}{n-1}}.
$$
Here, and thereafter, $B_1:=B({\bf0},1)$ denotes
the \emph{unit ball} in $\mathbb{R}^n$ center at $\mathbf{0}$,
and $\mathbf{0}$ denotes the \emph{origin} of $\mathbb{R}^{n}$.

In this article, we find the sharp constant $C$ in the isoperimetric inequality
in \eqref{e1.1} in the planar case, which  hence answers Mazya's problem \cite[Problem 1, p.\,5]{m18} in this case.

\begin{theorem}\label{thm:main}
Let $s\in(0,1)$ and $\Omega\subset\mathbb{R}^2$ be a bounded domain with $C^1$ boundary. Then
\begin{align}\label{eq:main}
P_s(\Omega)\le \frac{P_s(B_1)}{(2\pi)^{2-s}}\left[\mathcal{H}^1(\partial\Omega)\right]^{2-s},
\end{align}
where
\begin{align*}
P_s(B_1)=\frac{2^{2-s }\pi^{\frac{3}{2}}\Gamma\left(\frac{3-s}{2}\right)}
  {s(1-s)\Gamma\left(\frac{4-s}{2}\right)}
\end{align*}
with  $\Gamma$ denoting the Gamma function.
Moreover, when $\Omega=B_1$, the equality in \eqref{eq:main} holds .
\end{theorem}

\begin{remark}
By Theorem \ref{thm:main} and \cite[Section 6]{m03}, we conclude that the sharp constant in the
fractional Sobolev inequality \eqref{e1.1x} is
$$
\left[\frac{2\pi^{s-\frac{1}{2}}\Gamma\left(\frac{3-s}{2}\right)}
  {s(1-s)\Gamma\left(\frac{4-s}{2}\right)}\right]^{\frac{1}{2-s}}.
$$
\end{remark}

We now describe the proof strategy for Theorem \ref{thm:main}. The proof of Theorem \ref{thm:main}
proceeds through a series of reductions and sharp geometric inequalities.

First, we compare the fractional $s$-perimeter $P_s(\Omega)$ with a \emph{chord functional} $C_{1-s}(\Omega)$ (see Proposition
\ref{prop:slicing11}), which measures the total $(1-s)$-th power
of chord lengths in all directions.
We show that, for general domains, the inequality
\begin{align}\label{e1.3}
P_s(\Omega)\le\frac{1}{s(1-s)}C_{1-s}(\Omega)
\end{align}
holds and, for convex bodies, the equality in \eqref{e1.3} holds.
This chord functional formulation
transforms the nonlocal perimeter into a geometrically more tractable object.

Second, we reduce the problem from connected domains to planar convex bodies by showing that replacing a domain by
the closure of its convex hull does not increase the chord functional beyond an error controlled by the difference
in perimeters. This reduction, which relies on Crofton's formula and a careful slicing argument,
is essential for the proof.

Third, we prove the \emph{sharp convex chord inequality}: for any planar convex body $K$ and $q \in (0,1)$,
\begin{align}\label{e1.2}
C_q(K) \le C_q(B_1)\left[\frac{\mathrm{Per}(K)}{2\pi}\right]^{q+1},
\end{align}
where $C_q(K)$ denotes the chord functional in \eqref{eq:2.3x}
and
$\mathrm{Per}(K)$ denotes the \emph{perimeter} of $K$ in $\mathbb{R}^2$ (see Theorem \ref{thm:conv}).
The proof of \eqref{e1.2} is the most substantial part of this article. It combines a \emph{fractional Willmore-type
inequality} for convex curves, a first variation formula for the chord functional along outer parallel bodies,
and an asymptotic analysis as the parallel bodies expand.
Finally, using \eqref{e1.3} and \eqref{e1.2}, we prove \eqref{eq:main} and hence complete
the proof of Theorem \ref{thm:main}.

The remainder of this article is organized as follows.

In Section \ref{sec:chord}, we reduce the sharp fractional isoperimetric
inequality to a sharp inequality for a chord functional. Specifically, in
Subsection \ref{sec:chordfunctional}, we recall the chord functional and
prove a  formula which compares the fractional perimeter with this
functional. In Subsection \ref{sec:conn}, we reduce the problem from connected
domains to planar convex bodies.

In Section \ref{sec:convchord}, we prove the sharp convex
chord inequality (see Theorem \ref{thm:conv}).
We begin by establishing Theorem \ref{thm:conn} as a
consequence of Theorem \ref{thm:conv}.
Following this, we proceed to the proof of Theorem \ref{thm:main}.
In Subsection \ref{sec:comparison}, we use the fractional Willmore-type
inequality (see Lemma \ref{lem:FW}) to prove a derivative comparison
for the chord functional. In Subsection \ref{sec:limit}, we obtain the asymptotic behavior of
the chord functional along outer parallel bodies. Finally, in Subsection
\ref{sec:proof}, we complete the proof of
Theorem \ref{thm:conv}.

We restrict ourselves to the planar case because several essential steps of
our proof rely on geometric features that are special to dimension two.
First, for a connected planar set, its orthogonal projection onto any line is
an interval. This property, together with the fact that taking the convex hull
does not increase the perimeter in the plane, allows us to control the error
produced when the original set is replaced by its convex hull in the estimate
of the fractional perimeter. In higher dimensions, projections of connected
sets onto hyperplanes need not be convex, and the surface area of the convex
hull cannot, in general, be controlled by that of the original set.

Even for convex sets, the remaining argument is essentially planar. The
boundary of a planar convex body is a closed curve that can be described by a
single angular parameter and the corresponding supporting lines. In higher
dimensions, the boundary is a hypersurface rather than a curve, supporting
lines are replaced by supporting hyperplanes, and their normal directions can
no longer be described by a single scalar angle. As a result, the
one-dimensional identities used in the planar argument have no direct
higher-dimensional analogue.

Finally, the last part of the proof uses the particularly simple growth of
the outer parallel bodies $K+rB_1$. In the plane, the perimeter of $K+rB_1$
is a linear function of $r$, namely
$\operatorname{Per}(K+rB_1)
=\operatorname{Per}(K)+2\pi r.$
Thus, after normalizing $\operatorname{Per}(K)=2\pi$, the sets $K+rB_1$ and
$B_{r+1}$ have the same perimeter for every $r>0$, which makes it possible to
compare their chord functionals directly. In higher dimensions, the surface
area of $K+rB_1$ is instead a polynomial in $r$, whose coefficients depend on
additional geometric quantities of $K$, such as its curvature integrals.
Therefore, even if $K$ and a ball have the same surface area initially, their
parallel bodies generally have different surface areas for $r>0$, and the
comparison argument used in the plane no longer applies directly.

At the end of this section, we make some notational conventions. Let
${\mathbb N}:=\{1,2,\ldots\}$ and ${\mathbb Z}_+:={\mathbb N}\cup\{0\}$.
We always denote by $C$ a positive constant
which is independent of the main parameters involved,
but it may vary from line to line.
We also use $C_{\alpha,\beta,\ldots}$ to denote a
positive constant depending on the indicated parameters
$\alpha$, $\beta,\ldots$.
The notation $f\lesssim g$ means that $f\le Cg$.
If $f\lesssim g$ and $g\lesssim f$,
we then write $f\sim g$. If $f\le Cg$ and $g=h$
or $g\le h$, we then write $f\lesssim g=h$ or $f\lesssim g\le h$.
For any $x\in{\mathbb R}^2$ and $r\in(0,\infty)$, we denote by
$B(x,r):=\{y\in{\mathbb R}^2:\ |y-x|<r\}$
the ball with center $x$ and radius $r$; for simplicity,
we write $B_r:=B(\mathbf{0},r)$. We also use the \emph{notation} $\mathbb{S}^1:=\{x\in\mathbb{R}^2: |x|=1\}$
to denote the \emph{unit circle} in $\mathbb{R}^2$. For any $u\in\mathbb{S}^1$,
let $u^\perp:=\{z\in\mathbb R^2:\ z\cdot u=0\}$.
For any set $E\subset\mathbb R^2$, we use $\overline E$,
$\operatorname{int}E$, and $E^\complement$ to denote, respectively,
its closure, interior, and complement in $\mathbb R^2$.
We use $\mathbf{1}_E$ to denote the characteristic function of a
measurable set $E\subset {\mathbb R}^2$.
For any $x\in\mathbb R^2$ and any nonempty set
$E\subset\mathbb R^2$, let
$$
\operatorname{dist}(x,E):=\inf\{|x-y|:\ y\in E\}\ \ \text{and}\ \
\operatorname{diam}(E):=\sup\{|x-y|:\ x,y\in E\}.
$$
Moreover, the \emph{notation} $r\to0^+$ means that $r\in(0,\infty)$ and $r\to0$ and
the \emph{notation} $O(r)$ means that $\lim_{r\to\infty}\frac{r}{O(r)}$
exists and is finite. Finally, in all proofs we consistently retain the notation
introduced in the original theorem (or related statement).

\section{From Fractional Perimeter to the Chord Functional}\label{sec:chord}

This section consists of two subsections. Specifically,
in Subsection \ref{sec:chordfunctional} we introduce the chord functional
and prove a  comparison between   this
functional and the fractional perimeter. We also compute the value of the chord functional for the unit
disk, which fixes the sharp constant needed later.
In Subsection \ref{sec:conn}, we compare a connected smooth planar domain
with the closure of its convex hull. More precisely, using a perimeter
estimate and Crofton's formula, we reduce the desired chord inequality from
connected domains to planar convex bodies. This reduction will be used later in the
proof of Theorem \ref{thm:conn} and hence in the proof of Theorem
\ref{thm:main}.

\subsection{The Chord Functional}\label{sec:chordfunctional}

Let $\Omega\subset\mathbb{R}^2$ be a bounded open set,  $u\in\mathbb{S}^1$, and
$z\in u^\perp:=\{z\in\mathbb{R}^2:\ z\cdot u=0\}.$
Define
\begin{align}\label{zhixian}
\ell(u,z):=z+\mathbb{R} u:=\{z+tu:
t\in\mathbb{R}\}
\end{align}
and
\begin{align}\label{jiexian}
\Omega_{u,z}:=\{t\in\mathbb{R}:
z+tu\in \Omega\}.
\end{align}
Then it is easy to verify that
\begin{align*}
\Omega\cap \ell(u,z)
=\{z+tu:\ t\in\Omega_{u,z}\}.
\end{align*}
Moreover,  since $\Omega$ is open, it follows  that
$\Omega_{u,z}\subset \mathbb{R}$  is open. Thus,
$\Omega_{u,z}$
is a countable disjoint union of open intervals.  Let $\mathcal I(\Omega;u,z)$ denote the
family of its interval components. Then
\begin{align}\label{gao23}
\Omega_{u,z}=\bigcup_{J\in\mathcal I(\Omega;u,z)}J.
\end{align}
Let $q\in(0,1)$.
Define the \emph{chord functional} $C_q(\Omega)$
of $\Omega$ by setting
\begin{align}\label{eq:2.3x}
C_q(\Omega):=
\int_{\mathbb{S}^1}\int_{u^\perp}
\sum_{J\in\mathcal I(\Omega;u,z)} |J|^q\,dz\,d\sigma(u).
\end{align}
It is easy to check that, for any $\lambda\in (0,\infty)$,
\begin{align}\label{eq:shensuo}
C_q(\lambda \Omega)=\lambda^{q+1}C_q(\Omega).
\end{align}
Recall that a planar convex body is a compact convex subset in $\mathbb R^2$ with nonempty interior.
We also use $C_q$ to denote the chord functional of planar convex bodies. More precisely,
for any given planar convex body
$K$, let $C_q(K):=C_q(\operatorname{int}K)$,
where $\operatorname{int}K$ denotes the \emph{interior} of $K$.
The following proposition compares the fractional perimeter with the chord functional and
shows that the comparison is sharp for planar convex bodies.
\begin{proposition}\label{prop:slicing11}
Let $\Omega\subset\mathbb{R}^2$ be a bounded open set and $s\in(0,1)$.  Then
\begin{align}\label{eq:slice}
P_s(\Omega)\le \frac{1}{s(1-s)}C_{1-s}(\Omega).
\end{align}
Moreover, if $\Omega=\operatorname{int}K$ for some planar convex body $K$,
then the equality holds in \eqref{eq:slice}.
\end{proposition}

\begin{proof}
By polar decomposition, we find that
\begin{align*}
P_s(\Omega)
&=\int_\Omega\int_{\mathbb{R}^2\setminus \Omega}\frac{\,dx\,dy}{|x-y|^{2+s}}
=\int_{\mathbb{R}^2}\int_{\mathbb{R}^2}
\frac{\mathbf 1_\Omega(x)\mathbf 1_{\mathbb{R}^2\setminus \Omega}(x+h)}{|h|^{2+s}}\,dh\,dx\notag\\
&=\int_{\mathbb{S}^1}\int_0^\infty\int_{\mathbb{R}^2}
\mathbf 1_\Omega(x)\mathbf 1_{\mathbb{R}^2\setminus \Omega}(x+ru)
r^{-1-s}\,dx\,dr\,d\sigma(u).
\end{align*}
Observe that, for any $u\in\mathbb{S}^1$ and $r\in (0,\infty)$,
\begin{align*}
\int_{\mathbb{R}^2}\mathbf 1_\Omega(x)\mathbf 1_{\mathbb{R}^2\setminus \Omega}(x+ru)\,dx
&=\int_{u^\perp}\int_\mathbb{R}\mathbf 1_\Omega(z+tu)\mathbf 1_{\mathbb{R}^2\setminus \Omega}(z+tu+ru)\,dt\,dz\\
&=\int_{u^\perp}\int_\mathbb{R}\mathbf 1_{\Omega_{u,z}}(t)\mathbf 1_{\mathbb{R}\setminus \Omega_{u,z}}(t+r)
\,dt\,dz.
\end{align*}
Thus,
\begin{align}\label{eq:shalf}
P_s(\Omega)=\int_{\mathbb{S}^1}\int_{u^\perp}\int_0^\infty\int_\mathbb{R}
\mathbf 1_{\Omega_{u,z}}(t)\mathbf 1_{\mathbb{R}\setminus \Omega_{u,z}}(t+r)
r^{-1-s}\,dt\,dr\,dz\,d\sigma(u).
\end{align}
Let $A\subset\mathbb{R}$ be bounded and measurable.  Define
\begin{align*}
I_+(A):=\int_0^\infty\int_\mathbb{R}\mathbf 1_{A}(t)\mathbf 1_{\mathbb{R}\setminus A}(t+r)
r^{-1-s}\,dt\,dr
\end{align*}
and
\begin{align*}
I_-(A):=\int_0^\infty\int_\mathbb{R}\mathbf 1_{A}(t)\mathbf 1_{\mathbb{R}\setminus A}(t-r)
r^{-1-s}\,dt\,dr.
\end{align*}
Then it is easy to verify that $I_+(A)=I_-(A)$ and
\begin{align*}
\int_A\int_{\mathbb{R}\setminus A}\frac{\,dt\,d\tau}{|t-\tau|^{1+s}}=2I_+(A).
\end{align*}
Using this and \eqref{eq:shalf}, we find that
\begin{align}\label{eq:sliced}
P_s(\Omega)=\frac12\int_{\mathbb{S}^1}\int_{u^\perp}
\int_{\Omega_{u,z}}\int_{\mathbb{R}\setminus \Omega_{u,z}}
\frac{\,dt\,d\tau}{|t-\tau|^{1+s}}\,dz\,d\sigma(u).
\end{align}
By \eqref{gao23}, we conclude that,  for any $u\in\mathbb{S}^1$ and $z\in u^\perp$,
\begin{align}\label{yin2.7}
\int_{\Omega_{u,z}}\int_{\mathbb{R}\setminus \Omega_{u,z}}
\frac{\,dt\,d\tau}{|t-\tau|^{1+s}}
&=\sum_{J\in\mathcal I(\Omega;u,z)}
\int_J\int_{\mathbb{R}\setminus \Omega_{u,z}}
\frac{\,dt\,d\tau}{|t-\tau|^{1+s}}  \notag\\
&\le\sum_{J\in\mathcal I(\Omega;u,z)}\int_J\int_{\mathbb{R}\setminus J}
\frac{\,dt\,d\tau}{|t-\tau|^{1+s}}=\sum_{J\in\mathcal I(\Omega;u,z)}
\frac{2|J|^{1-s}}{s(1-s)}.
\end{align}
From this and  \eqref{eq:sliced}, we infer that \eqref{eq:slice} holds.

Finally, observe that, for any $x\in \operatorname{int} K$, $y\in K$,
and $\lambda\in [0,1),$ $(1-\lambda)x+\lambda y\in \operatorname{int} K.$
This
implies that, for any $u\in\mathbb{S}^1$ and $z\in u^\perp,$
$(\operatorname{int}K)_{u,z}$ is  an open interval or emptyset. Thus,
the equality in \eqref{yin2.7} holds. This finishes the proof of Proposition  \ref{prop:slicing11}.
\end{proof}

\begin{lemma}\label{lem:disk}
Let $q\in(0,1)$.  Then
\begin{align}\label{eq:Cqdisk}
C_q(B_1)=2^{q+1}\pi\int_{-1}^{1}(1-t^2)^{\frac q2}\,dt
=\frac{2^{q+1}\pi^{\frac{3}{2}}\Gamma(1+\frac{q}{2})}{\Gamma(\frac{3}{2}+\frac{q }{2})}.
\end{align}
Moreover,
\begin{align}\label{eq:dlb}
C_q(B_1)> \frac{4\pi^{q+1}}{q+1}.
\end{align}
\end{lemma}
\begin{proof}
Since  $B_1$ is a convex body, it follows that
\begin{align*}
C_q(B_1)=\int_{\mathbb{S}^1}\int_{u^\perp}|B_1\cap \ell(u,z )|^q\,dz\,d\sigma(u).
\end{align*}
Recall that, for any $a,b\in(0,\infty)$, the beta function $B(a,b)$ is defined by
\begin{align*}
B(a,b):=\int_0^1 r^{a-1}(1-r)^{b-1}\,dr.
\end{align*}
Moreover, it satisfies the identity
\begin{align*}
B(a,b)=\frac{\Gamma(a)\Gamma(b)}{\Gamma(a+b)},
\end{align*}
where $\Gamma$ denotes the Gamma function. Observe that, for any $u\in\mathbb{S}^1$,
\begin{align*}
\int_{u^\perp}|B_1\cap\ell(u,z )|^q\,dz&=
\int_{-1}^{1}\left(2\sqrt{1-t^2}\right)^q\,dt
=2^q\int_{-1}^1(1-t^2)^{\frac{q}{2}}\,dt\\
&=2^q \int_0^1 r^{-\frac{1}{2}}(1-r)^{\frac{q}{2}}\,dr
=2^qB\left(\frac12,1+\frac q2\right).
\end{align*}
Thus,
\begin{align*}
C_q(B_1)=2^{q+1}\pi B\left(\frac12,1+\frac q2\right)
=\frac{2^{q+1}\pi^{\frac{3}{2}}\Gamma(1+\frac{q}{2})}{\Gamma(\frac{3}{2}+\frac{q }{2})}.
\end{align*}
This proves \eqref{eq:Cqdisk}.

It remains to show \eqref{eq:dlb}.
For any $p\in\mathbb{R},$ let
$W_p:=\int_0^{\frac{\pi}{2}}\sin^p\theta\,d\theta.$
From the change of variables $t=\cos\theta$, we deduce that
\begin{align*}
\int_{-1}^1(1-t^2)^{\frac{q}{2}}\,dt=
\int_0^\pi \sin^{q+1}\theta\,d\theta=2W_{q+1}
\end{align*}
and hence $C_q(B_1)=2^{q+2}\pi W_{q+1}.$
Using integration by parts, we conclude that, for any $p\in (1,\infty),$
\begin{align*}
W_p&=-\int_0^{\frac{\pi}{2}}\sin^{p-1}\theta\,d\cos\theta
=(p-1)\int_0^{\frac{\pi}{2}}\sin^{p-2}\theta
\left(1-\sin^2\theta\right)\,d\theta   \\
&=(p-1)(W_{p-2}-W_p).
\end{align*}
This, together with the assumption that $q<1$,  implies that
\begin{align*}
C_q(B_1)&=2^{q+2}\pi W_{q+1}=\frac{2^{q+2}q\pi}{q+1}W_{q-1}
=\frac{2^{q+2}q\pi}{q+1}\int_0^{\frac{\pi}{2}}\sin^{q-1}\theta\,d\theta\\
&>\frac{2^{q+2}q\pi}{q+1}\int_0^{\frac{\pi}{2}}\theta^{q-1}\,d\theta=\frac{4\pi^{q+1}}{q+1}.
\end{align*}
This finishes the proof of Lemma \ref{lem:disk}.
\end{proof}

\subsection{Reduction from Connected Domains to Convex Bodies}\label{sec:conn}

Let $E\subset\mathbb{R}^2$ be a measurable set. Recall that the \emph{perimeter} $\mathrm{Per}(E)$ of $E$
is defined by setting
\begin{equation*}
\mathrm{Per}(E):=\sup\left\{\int_E \operatorname{div}\varphi(x)\,dx:
\varphi\in C_{\rm c}^1(\mathbb{R}^2;\mathbb{R}^2),\ |\varphi(x)|\le 1\ \text{for any }x\in\mathbb{R}^2\right\},
\end{equation*}
where $C_{\rm c}^1(\mathbb{R}^2;\mathbb{R}^2)$ denotes the space of all  continuously differentiable vector
fields from $\mathbb{R}^2$ to $\mathbb{R}^2$ with compact support.
Moreover,  $E$ is said to have
 \emph{finite perimeter} if $\mathrm{Per}(E)<\infty$. We have the following
conclusion on $\mathrm{Per}(E)$ and $\mathcal{H}^1(\partial E)$ (see, for instance,
\cite[Theorem 9.3]{m12} and \cite[Section 2.2]{s14}).

\begin{lemma}\label{lem2.1}
Let $E\subset\mathbb{R}^2$ be a bounded domain with $C^1$ boundary.  Then $E$ has finite perimeter and
$\mathrm{Per}(E)=\mathcal{H}^1(\partial E).$
Moreover, let $K\subset\mathbb{R}^2$ be a planar convex body, i.e., a compact convex set with nonempty interior.
Then $K$ has finite perimeter and
$\mathrm{Per}(K)=\mathcal{H}^1(\partial K).$
\end{lemma}

Furthermore, let $\Omega\subset\mathbb{R}^2$ be a bounded domain (connected open set) and let  $\operatorname{conv}(\Omega)$ be
the intersection of all the convex subsets of $\mathbb{R}^2$ that contain $\Omega$.
Denote the closure of  $\operatorname{conv}(\Omega)$ by $\overline{\operatorname{conv}}(\Omega)$.
Then $\overline{\operatorname{conv}}(\Omega)$ is compact and has nonempty interior.
The following  proposition is the main result of this subsection.
\begin{proposition}\phantomsection\label{prop:excess}
Let $\Omega\subset\mathbb{R}^2$ be a bounded  domain with $C^1$ boundary, $K:=
\overline{\operatorname{conv}}(\Omega)$, and  $q\in(0,1)$. Then
\begin{align}\label{eq:hullper}
\operatorname{Per}(K)\le \operatorname{Per}(\Omega)
\end{align}
and
\begin{align}\label{eq:excess}
C_q(\Omega)-C_q(K)\leq
2[\operatorname{diam}(K)]^q\left[\operatorname{Per}(\Omega)-\operatorname{Per}(K)\right].
\end{align}
\end{proposition}
We first make some preparations for the proof of Proposition \ref{prop:excess}.
The following Crofton intersection formula for rectifiable plane curves will be used below
(see, for instance, \cite[Theorem 5]{ad97} and \cite[(3.17)]{s04}).
\begin{lemma}\label{lem:crofton}
Let $\Gamma\subset\mathbb{R}^2$ be a rectifiable  curve.
For any $u\in\mathbb{S}^1$ and $z\in u^\perp$, let
\begin{align*}
M_\Gamma(u,z):=\#(\ell(u,z)\cap\Gamma),
\end{align*}
where $\ell(u,z)$ is the same as in \eqref{zhixian} and $\#(\ell(u,z)\cap\Gamma)$
denotes the number of intersection points of the two sets $\ell(u,z)$ and $\Gamma$. Then
\begin{align*}
\int_{\mathbb{S}^1}\int_{u^\perp}M_\Gamma(u,z)\,dz\,d\sigma(u)
=4\mathcal{H}^1(\Gamma).
\end{align*}
\end{lemma}
Using Lemma \ref{lem:crofton}, we obtain the following conclusion.

\begin{lemma}\label{lem:count}
The following assertions hold.
\begin{itemize}
\item [$\mathrm{(i)}$] Let $\Omega\subset\mathbb{R}^2$
be a bounded domain with $C^1$ boundary.
For any $u\in\mathbb{S}^1$ and
$z\in u^\perp$, let $N_\Omega(u,z)\in\mathbb{Z}_+\cup\{\infty\}$ denote the number of
interval components of the open subset $\Omega_{u,z}\subset\mathbb{R}$
defined in \eqref{jiexian}.
Then
\begin{align}\label{eq:jifenOmega}
\int_{\mathbb{S}^1}\int_{u^\perp}N_\Omega(u,z)\,dz\,d\sigma(u)=2\operatorname{Per}(\Omega).
\end{align}
\item[$\mathrm{(ii)}$] Let  $K$ be  a planar convex body. Then
\begin{align}\label{eq:jifenK}
\int_{\mathbb{S}^1}\int_{u^\perp}N_K(u,z)\,dz\,d\sigma(u)=2\operatorname{Per}(K),
\end{align}
where, for any $u\in\mathbb{S}^1$ and $z\in u^\perp$,
\begin{align*}
N_K(u,z)=\begin{cases}1, & \ell(u,z)\cap K\ne\emptyset,\\
0, & \ell(u,z)\cap K=\emptyset.
\end{cases}
\end{align*}
\end{itemize}
\end{lemma}

\begin{proof}
We first prove \eqref{eq:jifenOmega}.
Using Lemma \ref{lem:crofton},  we conclude that, to show \eqref{eq:jifenOmega}, it
suffices to prove that, for any $u\in\mathbb{S}^1$ and almost every $z\in u^\perp,$
\begin{align}\label{eq:jiaodianOmega}
M_{\partial\Omega}(u,z)=2N_\Omega(u,z).
\end{align}
Let $u:=(u_1,u_2)\in\mathbb{S}^1$ and  $q_u:=(u_2,-u_1)$. Then $q_u$ spans
$u^\perp$, and any $z\in u^\perp$ can be written uniquely as $z=\zeta q_u$ for some
$\zeta\in\mathbb{R}$. Let
$\partial\Omega=\Gamma_1\cup\cdots\cup\Gamma_N,$
where  $\{\Gamma_i\}_{i=1}^N $ are  $C^1$ Jordan curves.
Let $i\in\{1,\ldots,N \}$ and let $\gamma_i:\ [0,1]\to \Gamma_i$  be
a  $C^1$ parametrization of $\Gamma_i$. For any $s\in [0,1]$, let
$g_i(s):=\gamma_i(s) q_u.$
By Sard's theorem, for almost every $\zeta\in\mathbb{R}$, the number $\zeta$
is a regular value of $g_i$. For such $\zeta$, with $z:=\zeta q_u$, the
line $\ell(u,z)$ intersects $\partial\Omega$ only in finitely many points
and all these intersections are transversal.
Let these intersection points be ordered along the line such that
\begin{align*}
  \ell(u,z)\cap\partial\Omega
  =
  \{z+t_1u,\ldots,z+t_m u\}
\end{align*}
with $t_1<\cdots<t_m.$
Observe that, on each interval component of
$
\mathbb{R}\setminus\{t_1,\ldots,t_m\},
$
the function
$
t\mapsto \mathbf{1}_\Omega(z+tu)
$
is constant. Moreover, by transversality and the assumption that $\partial\Omega$ is $C^1$,
it is easy to verify  that
this constant changes when $t$
passes through each $t_j$.
Thus,  the values of $\mathbf{1}_\Omega(z+tu)$ alternate between $0$ and $1$,
starting and ending with $0$. Consequently,  $m$ is even and
\begin{align*}
  \Omega_{u,z}
  =
  (t_1,t_2)\cup(t_3,t_4)\cup\cdots\cup(t_{m-1},t_m).
\end{align*}
This implies that  \eqref{eq:jiaodianOmega}  holds.

Since $K$  is a planar convex  body, it follows that  $\partial K$ is a rectifiable Jordan curve
(see, for instance, \cite[p. 15]{t06} and \cite[Theorem 32]{lrs95}).
Using this and Lemma \ref{lem:crofton}, we find that, to show \eqref{eq:jifenK}, it  suffices to prove that,
for any $u\in\mathbb{S}^1$ and almost every $z\in u^\perp,$
\begin{align}\label{eq:jiaodianK}
M_{\partial K}(u,z)=2N_K(u,z).
\end{align}
Let $u\in\mathbb{S}^1$ and
\begin{align}\label{eq:yibian}
Z_u:=P_u^\perp(K)\setminus \operatorname{int}_{u^\perp}P_u^\perp(K),
\end{align}
where the interior is taken relative to the line $u^\perp$. Since
$P_u^\perp(K)$ is a compact interval in $u^\perp$,
it follows that  the set $Z_u$ consists of at most two points.
Let $z\in u^\perp\setminus Z_u$.  If $z\notin P_u^\perp(K)$, then
$\ell(u,z)\cap K=\emptyset$ and hence $N_K(u,z)=M_{\partial K}(u,z)=0$, which implies
that \eqref{eq:jiaodianK} holds in this case.
Let $z\in \operatorname{int}_{u^\perp}P_u^\perp(K).$
Then $N_K(u,z)=1.$ Moreover, the line $\ell(u,z)$ intersects $K$ in a nondegenerate compact interval. Thus,
$M_{\partial K}(u,z)=2$ and \eqref{eq:jiaodianK} holds in this case.
This finishes the proof of Lemma \ref{lem:count}.
\end{proof}
For any given $u\in\mathbb{S}^1$ and any $x\in\mathbb{R}^2$, let
$P_u^\perp x:=x-(x\cdot u)u.$ Now, we prove Proposition \ref{prop:excess}.
\begin{proof}[Proof of Proposition \ref{prop:excess}]
We first show \eqref{eq:excess}.
For any $u\in\mathbb{S}^1$ and $z\in u^\perp$, define
\begin{align*}
S_\Omega(u,z):=\sum_{I\in\mathcal I(\Omega;u,z)} |I|^q\ \ \text{and}\ \ K_{u,z}:=\{t\in\mathbb{R}:
z+tu\in K\},
\end{align*}
where $\mathcal I(\Omega;u,z)$ is the same as in \eqref{gao23}. Then
\begin{align*}
C_q(\Omega)=\int_{\mathbb{S}^1}\int_{u^\perp}S_\Omega(u,z)\,dz\,d\sigma(u).
\end{align*}
We claim that, for any $u\in\mathbb{S}^1$ and  almost every
$z\in u^\perp$,
\begin{align}\label{eq:olex}
S_\Omega(u,z)\le |K_{u,z}|^q+\left[N_\Omega(u,z)-N_K(u,z)\right][\operatorname{diam}(K)]^q.
\end{align}
If this claim  holds, then, by the definition of $C_q(K )$ and
Lemma \ref{lem:count}, we conclude that
\begin{align*}
C_q(\Omega)&\le\int_{\mathbb S^1}\int_{u^\perp}|K_{u,z}|^q\,dz\,d\sigma(u)+[\operatorname{diam}(K)]^q\int_{\mathbb S^1}\int_{u^\perp}
\left[N_\Omega(u,z)-N_K(u,z)\right]\,dz\,d\sigma(u)\\
&=C_q(K)+2[\operatorname{diam}(K)]^q
\left[\operatorname{Per}(\Omega)-\operatorname{Per}(K)\right].
\end{align*}
This proves \eqref{eq:excess}.

Now, we show \eqref{eq:olex}. Let $u\in\mathbb S^1$.
Since $\Omega$ is connected,  it follows that
$P_u^\perp(\Omega)$ is an interval in $u^\perp$ and
$P_u^\perp(K)=\overline{P_u^\perp(\Omega)}$.
Let $Z_u$ be the same as in \eqref{eq:yibian}.
We now prove \eqref{eq:olex} for any $z\in u^\perp\setminus Z_u$.
If $z\notin P_u^\perp(K)$, then
$\ell(u, z)\cap K=\emptyset$. This implies that $S_\Omega(u,z)=0,$ $N_\Omega(u,z)=N_K(u,z)=0,$ and $|K_{u,z}|=0.$
Thus, \eqref{eq:olex} holds
in this case. Assume that $z\in \operatorname{int}_{u^\perp}P_u^\perp(K).$
Then $z\in P_u^\perp(\Omega)$ and
$\ell(u, z)\cap K\ne\emptyset$.
By this, we find that $\ell(u, z)\cap\Omega\ne\emptyset$ and
$N_\Omega(u,z)\ge 1$. Since $\Omega\subset K$,
it follows that, for any $I\in\mathcal I(\Omega;u,z),$
\begin{align}\label{eq:baohanK}
|I|\le |K_{u,z}|\le \operatorname{diam}(K).
\end{align}
Choose one component $I_0\in\mathcal I(\Omega;u,z)$.
From \eqref{eq:baohanK}, we deduce that
\begin{align*}
S_\Omega(u,z)=|I_0|^q+\sum_{I\in\mathcal I(\Omega;u,z)\setminus\{I_0\}} |I|^q
\leq|K_{u,z}|^q+\left(N_\Omega(u,z)-1\right)[\operatorname{diam}(K)]^q.
\end{align*}
This shows \eqref{eq:olex}.

Finally, as in the proof of \eqref{eq:olex}, we find that, for any $u\in\mathbb{S}^1$ and  almost every
$z\in u^\perp$, $N_\Omega(u,z)\ge N_K(u,z)$.
This, together  with Lemma \ref{lem:count}, implies that \eqref{eq:hullper} holds,
which  completes the proof of Proposition \ref{prop:excess}.
\end{proof}

\section{The Sharp Convex Chord Inequality}\label{sec:convchord}
In this section, we  prove the following sharp convex chord inequality.
\begin{theorem}\label{thm:conv}
Let   $K\subset\mathbb{R}^2$ be a planar convex body and $q\in(0,1)$.  Then
\begin{align}\label{eq:conv}
C_q(K)\le C_q(B_1)\left[\frac{\operatorname{Per}(K)}{2\pi}\right]^{q+1}.
\end{align}
Moreover, if $K:=\overline{B}_1$,
then the equality in \eqref{eq:conv} holds.
\end{theorem}
Assume for the moment that Theorem \ref{thm:conv} holds; using it
and
Proposition \ref{prop:excess}, we show the following conclusion, which is essential for
the proof of Theorem \ref{thm:main}.
\begin{theorem}\label{thm:conn}
Let $\Omega\subset\mathbb{R}^2$ be a bounded  domain with $C^1$ boundary and   $q\in(0,1)$. Then
\begin{align}\label{eq:omega}
C_q(\Omega)\le C_q(B_1)\left[\frac{\operatorname{Per}(\Omega)}{2\pi}\right]^{q+1}.
\end{align}
Moreover, if $\Omega:=B_1$, then the equality in \eqref{eq:omega} holds.
\end{theorem}

\begin{proof}
By \eqref{eq:shensuo}, without loss of generality, we may assume that $\operatorname{Per}(\Omega)=2\pi.$
Then, to obtain the desired conclusion, we just need to prove
\begin{align}\label{zuihou}
C_q(\Omega)\leq C_q(B_1).
\end{align}
Let $K:=\overline{\operatorname{conv}}(\Omega)$.
As is well known,  $\operatorname{diam}(K)\le \operatorname{Per}(K)/2$.
Using this, Proposition \ref{prop:excess}, and Theorem \ref{thm:conv},
we find that
\begin{align}\label{eq:CqO}
C_q(\Omega)
&\leq C_q(K)+2[\operatorname{diam}(K)]^q\left[\operatorname{Per}(\Omega)-\operatorname{Per}(K)\right]  \notag\\
&\le C_q(B_1)\left[\frac{\operatorname{Per}(K)}{2\pi}\right]^{q+1}
+2\left[\frac{\operatorname{Per}(K)}{2}\right]^q[2\pi-\operatorname{Per}(K)]\notag\\
&=C_q(B_1)t^{q+1}+4\pi^{q+1}t^q(1-t),
\end{align}
where $t:=\frac{\operatorname{Per}(K)}{2\pi}.$
By Proposition \ref{prop:excess}, we have $t\in (0,1]$.
If $t=1$, then, by \eqref{eq:CqO},  we conclude that \eqref{zuihou} holds. If $t\in (0,1),$ we have
\begin{align*}
1-t^{q+1}=(q+1)\int_t^1 r^q\,dr>(q+1)t^q(1-t).
\end{align*}
From this, \eqref{eq:CqO}, and Lemma \ref{lem:disk}, it follows that
\begin{align*}
C_q(\Omega)\leq C_q(B_1)t^{q+1}+4\pi^{q+1}t^q(1-t)
<C_q(B_1)[t^{q+1}+(q+1)t^q(1-t)]<C_q(B_1),
\end{align*}
and hence \eqref{zuihou} in this case holds.
This finishes the proof of Theorem \ref{thm:conn}.
\end{proof}

Now, using Theorem \ref{thm:conn}
and Proposition \ref{prop:slicing11},
we prove Theorem \ref{thm:main}.
\begin{proof}[Proof of Theorem \ref{thm:main}]
By Proposition \ref{prop:slicing11} and Theorem \ref{thm:conn}, we have
\begin{align}\label{eq:suncuo}
P_s(\Omega)\leq\frac{1}{s(1-s)}C_{1-s}(\Omega)
\leq\frac{1}{s(1-s)}C_{1-s}(B_1)
\left[\frac{\operatorname{Per}(\Omega)}{2\pi}\right]^{2-s}.
\end{align}
Moreover, from Proposition \ref{prop:slicing11} and Lemma \ref{lem:disk}, we infer that
\begin{align*}
P_s(B_1)=\frac{1}{s(1-s)}C_{1-s}(B_1)=\frac{
2^{2-s}\pi^{\frac{3}{2}}\Gamma\left(\frac{3-s}{2}\right)}
{s(1-s)\Gamma\left(\frac{4-s}{2}\right)}.
\end{align*}
This, together with \eqref{eq:suncuo} and Lemma \ref{lem2.1},  implies that
\eqref{eq:main} holds, which completes the proof of Theorem \ref{thm:main}.
\end{proof}
The remainder of this section is organized as follows.
Subsections \ref{sec:comparison}
and \ref{sec:limit}
are devoted to establishing
several estimates needed for the
proof of Theorem \ref{thm:conv},
whose proof is finally given in Subsection \ref{sec:proof}.

\subsection{A Derivative Comparison for the Chord Functional}\label{sec:comparison}
We say that a simple closed plane curve is \emph{positively oriented} if,
when one goes along the curve in the direction of increasing parameter, the interior of
the curve remains to the left. Let $\gamma$ be a plane curve parametrized by arc length.
The \emph{signed curvature} of $\gamma$ is defined by setting $\kappa_\gamma:=\det (\gamma',\gamma'').$
Recall that a compact convex set $K$ is said to be \emph{strictly convex} if, for any two distinct points $x,y\in \partial K$,
the open segment $(x,y):=\{(1-t)x+ty:0<t<1\}$ is contained in $\operatorname{int}K$.
Let  $k\in\{2,3,\ldots\}\cup\{\infty\}.$ A planar convex body $K$ is called a $C^k_+$ planar convex
body if $\partial K$ admits a positively  oriented $C^k$ arclength parametrization whose signed curvature
is positive. If $K$ is a $C^k_+$ planar convex body, then its boundary $\partial K$ is called a \emph{$C^k_+$ convex curve}.
Note that a $C^k_+$ planar convex body is automatically strictly convex.

Let $K\subset\mathbb{R}^2$ be a planar convex body and $\alpha\in(0,1)$.
Define the \emph{fractional Willmore-type quantity} by setting
\begin{align*}
W_\alpha(K):=\int_{\partial K}\int_{\partial K}\frac{(y-x)\cdot\nu_K(y)}{|x-y|^{2+\alpha}}
\,d\mathcal{H}^1(y)\,d\mathcal{H}^1(x),
\end{align*}
where $\nu_K$ denotes the outward unit normal, defined
$\mathcal{H}^1$-almost everywhere on $\partial K$.
It is easy to verify that
\begin{align*}
W_\alpha(B_1)=2^{1-\alpha}\pi^2\int_0^1\sin(\pi w)^{-\alpha}\,dw
\end{align*}
and, for any $\lambda\in(0,\infty)$,
\begin{align}\label{eq:kshensuo}
W_\alpha(\lambda K)=\lambda^{1-\alpha}W_\alpha(K).
\end{align}
For further results on fractional Willmore-type functionals, see
\cite{bgss25,gs24,dd26} and the references therein.
The following proposition is the main result of this subsection.

\begin{proposition}\label{prop:dcomp}
Let $K$ be a $C^\infty_+$ planar convex body,
$L:=\operatorname{Per}(K),$ and $q\in(0,1)$. Then
\begin{align}\label{eq:x56x}
\left.\frac{d}{dr}\right|_{r=0}C_q(K+rB_1)=2qW_{1-q}(K).
\end{align}
Consequently,
\begin{align}\label{eq:x56}
\left.\frac{d}{dr}\right|_{r=0}C_q(K+rB_1)\ge
\left.\frac{d}{dr}\right|_{r=0}C_q(B_{L/(2\pi)}+rB_1).
\end{align}
\end{proposition}
The rest of this subsection is devoted to the proof of Proposition \ref{prop:dcomp}.
Let $\Omega\subset\mathbb{R}^2$ be a bounded $C^2$ domain and  $\alpha\in(0,1)$.
Denote the \emph{fractional mean curvature} by setting, for any $x\in\partial \Omega,$
\begin{align*}
H_\alpha^\Omega(x):=\mathrm{p.v.}\int_{\mathbb{R}^2}
\frac{\mathbf 1_{\mathbb{R}^2\setminus\Omega}(y)-\mathbf 1_\Omega(y)}{|x-y|^{2+\alpha}}\,dy,
\end{align*}
where $\mathrm{p.v.}$ denotes the integral taken according to the Cauchy principal value sense.
We denote by $\nu_\Omega$ the outward unit normal along $\partial\Omega$.
The following first variation formula is a direct consequence of
\cite[Theorem 4.26]{l15} (see also \cite[Theorem 6.1]{ffmmm15}).
\begin{lemma}\label{lem:first}
Let $\Omega\subset\mathbb{R}^2$ be a bounded domain with $C^2$ boundary,
and let $\alpha\in(0,1)$. Let $r_0\in(0,\infty)$ and
$\Phi:(-r_0,r_0)\times\mathbb{R}^2\to\mathbb{R}^2$
be a map such that the following statements hold:

\begin{itemize}
\item[{\rm (i)}] If $\Phi_r(x):=\Phi(r,x)$, then $\Phi_0=\operatorname{Id}$.

\item[{\rm (ii)}] For any given $r\in(-r_0,r_0)$, $\Phi_r$ is a $C^2$
diffeomorphism of $\mathbb{R}^2$ onto itself.

\item[{\rm (iii)}] $\Phi$ is of class $C^2$.
\end{itemize}
Then
\begin{align*}
\left.\frac{d}{dr}P_\alpha(\Phi_r(\Omega))\right|_{r=0}
=\int_{\partial \Omega}H^\Omega_\alpha(x)\phi(x)\cdot\nu_\Omega(x)\,d\mathcal{H}^{1}(x),
\end{align*}
where
$\phi:=\left.\frac{\partial}{\partial r}\Phi_r\right|_{r=0}$.
\end{lemma}

In order  to  use Lemma \ref{lem:first}, we first prove the following  conclusion.

\begin{lemma}\label{lem:para}
Let $K\subset\mathbb{R}^2$ be a planar convex body with $C^\infty$ boundary.
Then there exist $r_0\in(0,\infty)$ and a map $\Phi:(-r_0,r_0)\times\mathbb{R}^2\to\mathbb{R}^2$
such that the following assertions hold:
\begin{itemize}
\item [$\mathrm{(i)}$] If $\Phi_r(z):=\Phi(r,z)$, then
$\Phi_0=\operatorname{Id}$.
\item [$\mathrm{(ii)}$] For any given $r\in (-r_0,r_0)$,  $\Phi_r$ is a $C^\infty$ diffeomorphism of
$\mathbb{R}^2$ onto itself.
\item [$\mathrm{(iii)}$] $\Phi$ is of class $C^\infty$.
\item [$\mathrm{(iv)}$] For any given $r\in [0,r_0)$,
\begin{align*}
\Phi_r(K)=K+r\overline{B}_1=\left\{z\in\mathbb{R}^2:\ \operatorname{dist}(z,K)\leq r\right\}.
\end{align*}
\item [$\mathrm{(v)}$] For any $x\in \partial K $,
\begin{align*}
\left.\frac{\partial}{\partial r}\Phi_r(x)\right|_{r=0}=\nu_K(x),
\end{align*}
where $\nu_{K}$ denotes the outward unit normal along $\partial K.$
\end{itemize}
\end{lemma}

\begin{proof}
For any $z\in\mathbb{R}^2,$ let
\begin{align*}
d_K(z):=\operatorname{dist}(z,K)
-\operatorname{dist}\left(z,\mathbb{R}^2\setminus K\right).
\end{align*}
Then $K=\{z\in\mathbb{R}^2:d_K(z)\le 0\}$ and $\partial K=\{z\in\mathbb{R}^2:d_K(z)=0\}$.
By the assumption that  $\partial K$ is compact and $C^\infty$ and by the tubular neighborhood theorem
(see, for instance, \cite[Theorem 6.24]{l13}), we conclude that there exists $\eta\in(0,\infty)$ such that
\begin{align*}
\Psi:\partial K\times(-\eta,\eta)\to \Psi(\partial K\times(-\eta,\eta))
\end{align*}
is a $C^\infty$ diffeomorphism, where, for any $(x,t)\in\partial K\times(-\eta,\eta)$,
$\Psi(x,t):=x+t\nu_K(x).$
We claim that, for any $(x,t)\in\partial K\times(-\eta,\eta)$,
\begin{align}\label{eq:signed}
\operatorname{dist}(x+t\nu_K(x),\partial K)=|t|.
\end{align}
If $t=0,$ then \eqref{eq:signed} is trivial. Assume that $t\not=0.$
Let $(x,t)\in\partial K\times(-\eta,\eta)$ and $z:=x+t\nu_K(x).$
Let $y\in\partial K$ be such that
$|z-y|=\operatorname{dist}(z,\partial K).$
Let $\gamma$ be a local $C^\infty$ parametrization of
$\partial K$ near $y$ with $\gamma(0)=y$. Since the function
$s\mapsto |z-\gamma(s)|^2$ has a local minimum at $s=0$, we deduce that
\begin{align*}
0=\left.\frac{d}{ds}|z-\gamma(s)|^2\right|_{s=0}=-2(z-y)\cdot\gamma'(0).
\end{align*}
Thus, $z-y$ belongs to the normal space of $\partial K$ at $y$ and hence
there exists $s\in\mathbb{R}$ such that
$z=y+s\nu_K(y).$
From this and the injectivity of $\Psi$, it follows that  $y=x$ and $s=t$. Thus,
\begin{align*}
\operatorname{dist}(x+t\nu_K(x),\partial K)
=\operatorname{dist}(z,\partial K)=|z-y|=|t|.
\end{align*}
This shows \eqref{eq:signed}.
Using \eqref{eq:signed}, we obtain, for any
$(x,t)\in\partial K\times(-\eta,\eta)$,
$d_K(x+t\nu_K(x))=t.$ Thus,
\begin{align*}
\Psi(\partial K\times(-\eta,\eta))=\{z\in\mathbb{R}^2:|d_K(z)|<\eta\}=:U_\eta.
\end{align*}
Choose $\chi\in C_{\rm c}^\infty((-\eta,\eta))$ such that $0\le \chi\le 1$, $\chi\equiv 1$
on $(-\frac{\eta}{2},\frac{\eta}{2}),$ and $\chi\equiv 0$
on $(-\infty,-\frac{3\eta}{4})\cup (\frac{3\eta}{4},\infty).$
Let $r_0\in(0,\min\{\frac{\eta}{4},\frac{1}{2\|\chi'\|_{L^\infty}}\})$ and $r\in(-r_0,r_0)$.
For any  $t\in (-\eta,\eta)$, define
 $f_r(t):=t+r\chi(t).$
Observe that $\lim_{t\to-\eta^+} f_r(t)=-\eta,$ $\lim_{t\to\eta^-} f_r(t)=\eta,$ and,
for any $t\in (-\eta,\eta)$,
\begin{align*}
f_r'(t)=1+r\chi'(t)\ge 1-|r|\,\|\chi'\|_{L^\infty((-\eta,\eta))}>\frac{1}{2}.
\end{align*}
This implies that
$f_r$ is strictly increasing on $(-\eta,\eta)$
and maps $(-\eta,\eta)$ bijectively onto $(-\eta,\eta)$.
Thus, $f_r$ is a $C^\infty$ diffeomorphism of $(-\eta,\eta)$ onto itself.

We define $\Phi:\ (-r_0,r_0)\times\mathbb{R}^2\to\mathbb{R}^2$ by
\begin{align*}
\Phi(r,z):=\begin{cases}
x+f_r(t)\nu_K(x) & \text{if } z=x+t\nu_K(x)\in U_\eta,\\
z & \text{if } z\in\mathbb{R}^2\setminus U_\eta.
\end{cases}
\end{align*}
Moreover,  let $\Phi_r(z):=\Phi(r,z)$. We now verify the property of $\Phi.$
By definitions of $\Phi$ and $f_r$, it is obvious that (i) holds.

Next, we prove that $\Phi_r$
is a $C^2$ map on $\mathbb{R}^2.$ For any $(x,t)\in\partial K \times (-\eta,\eta)$, let
$F_r(x,t):=(x,f_r(t))$. Then $F_r$ is a $C^\infty$ diffeomorphism of
$\partial K \times (-\eta,\eta)$ onto  itself. Observe that, in $U_\eta$,
the map $\Phi_r$ can be written as
$\Phi_r=\Psi\circ F_r\circ\Psi^{-1}.$
Since $\Psi$ is a $C^\infty$ diffeomorphism and $F_r$ is a $C^\infty$ diffeomorphism, it follows that
$\Phi_r$ is $C^\infty$ in $U_\eta$. On $\mathbb{R}^2\setminus U_\eta$, $\Phi_r$ is the identity map.
Thus, to obtain the desired conclusion, it suffices to
verify that the two definitions are compatible near $\partial U_\eta$.
By the definition of $\chi$, we find that, for any $z=x+t\nu_K(x)\in U_{\eta}$ with
$x\in\partial K$ and  $\frac{3\eta}{4}\leq |t|<\eta,$
\begin{align*}
\Phi_r(z)=x+f_r(t)\nu_K(x)=x+t\nu_K(x)=z.
\end{align*}
This implies that  $\Phi_r$ agrees with the identity map in the open set
\begin{align*}
\left\{x+t\nu_K(x):x\in\partial K,\ \frac{3\eta}{4}<|t|<\eta\right\},
\end{align*}
which is a neighborhood of $\partial U_\eta$. Consequently,
$\Phi_r$ is a  $C^\infty$ map on $\mathbb{R}^2$.
On the other hand, the inverse of $\Phi_r$ is given by
\begin{align*}
\Phi_r^{-1}(z):=\begin{cases}
x+f_r^{-1}(t)\nu_K(x), & \text{if } z=x+t\nu_K(x)\in U_\eta,\\
z, & \text{if } z\in\mathbb{R}^2\setminus U_\eta.
\end{cases}
\end{align*}
Similarly, $\Phi_r^{-1}$ is also a $C^\infty$ map on $\mathbb{R}^2.$
Thus,  $\Phi_r$ is a $C^\infty$ diffeomorphism of $\mathbb{R}^2$ onto itself. This shows (ii).

Since the map $(r,t)\mapsto f_r(t)$ is $C^\infty$, the same  argument shows that
$\Phi(r,z):\ (-r_0,r_0)\times\mathbb{R}^2\to\mathbb{R}^2$
is of class $C^\infty$, which proves (iii).

Let $r\in [0,r_0)$ and $A_r:=\{z\in\mathbb{R}^2:d_K(z)\le r\}$.
Then $A_r=K+r\overline{B}_1.$ Thus, to show (iv), it suffices to prove
$\Phi_r(K)=A_r.$
Let $z\in K$. If $z\notin U_\eta$, then $\Phi_r(z)=z\in K\subset A_r$. If $z\in U_\eta$,
then $z=x+t\nu_K(x)$ with $x\in\partial K $ and $t\in (-\eta,0]$. Since $f_r$ is increasing, it follows that
\begin{align*}
d_K(\Phi_r(z))=d_K(x+f_r(t)\nu_K(x))=f_r(t)\leq f_r(0)= r.
\end{align*}
This implies that $\Phi_r(z)\in A_r$ and hence  $\Phi_r(K)\subset A_r$.
Conversely, let $y\in A_r$. Then $d_K(y)\le r<r_0<\eta.$
Thus, if $y\notin U_\eta$, then $|d_K(y)|\ge\eta$ and hence
$d_K(y)\le -\eta<0$. By this, we have $y=\Phi_r(y)\in\Phi_r(K)$.
Now,  assume that $y\in A_r\cap U_\eta$. Write $y=x+s\nu_K(x)$
with $x\in\partial K$ and $s\in(-\eta,\eta)$. Then $s\le r=f_r(0)$.
Using this and the fact that $f_r$ is increasing, we conclude that there exists a
unique number $t:=f_r^{-1}(s)$ satisfies $t\le 0$. Thus, $z:=x+t\nu_K(x)\in K$, and
\begin{align*}
\Phi_r(z)=x+f_r(t)\nu_K(x)=x+s\nu_K(x)=y.
\end{align*}
This proves $A_r\subset\Phi_r(K)$, which completes the proof of (iv).

Finally, from definitions of $\Phi_r$ and $\chi,$ we infer that, for any  $x\in\partial K$,
\begin{align*}
\Phi_r(x)=x+f_r(0)\nu_K(x)=x+r\nu_K(x).
\end{align*}
Consequently, for any  $x\in\partial K$,
 $\frac{\partial}{\partial r}\Phi_r(x)|_{r=0}=\nu_K(x).$
This shows (v), which then completes  the proof of Lemma \ref{lem:para}.
\end{proof}

The following lemma is a special case of
\cite[Theorem 1.3]{dd26} with $p=1$.
\begin{lemma}\label{lem:FW}
Let $K\subset\mathbb{R}^2$ be a planar convex body and $\alpha\in(0,1)$. Then
\begin{align}\label{eq:FW}
W_\alpha(K)\geq W_\alpha(B_1)\left[\frac{\operatorname{Per}(K)}{2\pi}\right]^{1-\alpha}.
\end{align}
Moreover, equality holds in \eqref{eq:FW} if and only if $K$ is a disk.
\end{lemma}

Now, we prove Proposition \ref{prop:dcomp}.

\begin{proof}[Proof of Proposition \ref{prop:dcomp}]
Using Proposition \ref{prop:slicing11} and
Lemmas \ref{lem:first} and \ref{lem:para},
we conclude that, for any $q\in (0,1),$
\begin{align}\label{dong}
\left.\frac{d}{dr}\right|_{r=0}C_q(K+r\overline{B}_1)
&=q(1-q)\left.\frac{d}{dr}\right|_{r=0}P_{1-q}(K+r\overline{B}_1)\notag\\
&=q(1-q)\int_{\partial K}H_{1-q}^{\operatorname{int} K}(x)\,d\mathcal{H}^1(x).
\end{align}
By  the divergence theorem and the  dominated convergence theorem, we find that,
for any $x\in\partial K$,
\begin{align*}
H_{1-q}^{\operatorname{int}K}(x)
=\frac{2}{1-q}\int_{\partial K}
\frac{(y-x)\cdot\nu_K(y)}{|x-y|^{3-q}}
\,d\mathcal{H}^1(y).
\end{align*}
From this and \eqref{dong}, we infer that
\begin{align*}
\left.\frac{d}{dr}\right|_{r=0}C_q(K+r\overline{B}_1)
=2q\int_{\partial K}\int_{\partial K}
\frac{(y-x)\cdot\nu_K(y)}{|x-y|^{3-q}}
\,d\mathcal{H}^1(y)
\,d\mathcal{H}^1(x).
\end{align*}
This shows \eqref{eq:x56x}.

Finally, \eqref{eq:x56} follows from \eqref{eq:x56x}, \eqref{eq:kshensuo}, and Lemma \ref{lem:FW}.
This finishes the proof of Proposition \ref{prop:dcomp}.
\end{proof}

\subsection{Asymptotics along Outer Parallel Bodies}\label{sec:limit}
In this subsection, we show the following result.
\begin{proposition}\label{prop:lim}
 Let $K$ be a $C^2_+$ planar convex body with $\operatorname{Per}(K)=2\pi,$ and let $q\in(0,1)$.  Then
\begin{align}\label{eq:lim}
C_q(K+rB_1)-C_q(B_{r+1})=O(r^{q-1})\qquad\text{as }r\to\infty.
\end{align}
In particular,
\begin{align*}
\lim_{r\to\infty}\left[C_q(K+rB_1)-C_q(B_{r+1})\right]=0.
\end{align*}
\end{proposition}

To prove Proposition \ref{prop:lim}, we recall some  basic facts.
Let $K\subset\mathbb{R}^2$ be a planar convex body. The \emph{support function} $h_K$ is defined by setting,
for any $v\in\mathbb{S}^1$, $h_K(v):=\sup_{x\in K} x\cdot v.$
Let $v\in\mathbb{S}^1$. Define
\begin{align*}
\ell_K(v):=\{z\in\mathbb{R}^2:z\cdot v=h_K(v)\}
\ \ \text{and}\ \
F_K(v):=\{z\in K:z\cdot v=h_K(v)\}.
\end{align*}
Since $K$ is a planar convex body, it follows that $F_K(v)\neq\emptyset,$
\begin{align}\label{eq:yuanchuang}
K\subset\{z\in\mathbb{R}^2:z\cdot v\le h_K(v)\},
\ \ \text{and}\ \   K=\bigcap_{v\in\mathbb{S}^1}
\{z\in\mathbb{R}^2:z\cdot v\le h_K(v)\}.
\end{align}
Moreover, $\ell_K(v)$ is called the \emph{supporting line of $K$ with
outward normal $v$} and  $F_K(v)$ is called the \emph{exposed contact
set of $K$ in direction $v$}.
If, in addition, $\partial K$ is of class $C^1$, the supporting line
property of convex bodies (see \cite[Section 1.3]{s14}) yields,
for any $x\in\partial K$ and any $y\in K$,
\begin{align}\label{eq:x01}
(y-x)\cdot\nu_K(x)\leq0.
\end{align}
Recall that, for any given
nonempty compact sets $A,B\subset\mathbb{R}^2,$
their \emph{Hausdorff distance} $d_H(A,B)$ is defined by setting
\begin{align*}
d_H(A,B):=\max\left\{
\sup_{a\in A}\operatorname{dist}(a,B),
\sup_{b\in B}\operatorname{dist}(b,A)\right\}.
\end{align*}
The following standard facts on support functions and the Hausdorff distance
can be found in  \cite[Theorems 1.3.2 and  1.7.5 and Lemma  1.8.14]{s14}.
\begin{lemma}\label{lem:suppline}
Let $K,L\subset\mathbb{R}^2$ be planar convex bodies. Then the following statements  hold.
\begin{itemize}
\item  [$\mathrm{(i)}$]For any given $y\in\partial K,$ there exists $v\in\mathbb{S}^1$ such that
$y\in F_K(v).$
\item [$\mathrm{(ii)}$]For any $\lambda\in (0,\infty)$, $h_{K+L}=h_K+h_L$ and $h_{\lambda K}=\lambda h_{K}.$
\item  [$\mathrm{(iii)}$]
$d_H(K,L)=\sup_{v\in\mathbb{S}^1}|h_K(v)-h_L(v)|.$
\end{itemize}
\end{lemma}

For any $L\in(0,\infty)$, let $\mathbb{T}_L:=\mathbb{R}/(L\mathbb{Z})$.
We identify a function on $\mathbb{T}_L$ with its $L$-periodic representative
on $\mathbb{R}$, and understand its continuity, differentiability, and derivatives
through this representative.
We recall the following standard facts about support functions
of convex bodies with positive curvature; see, for instance,
\cite[Section 2.5]{s14} and \cite[Section 2.2]{g01}.

\begin{lemma}\label{lem:nangle}
Let $k\in\{2,3,\ldots\}\cup\{\infty\}$. For any $\theta\in\mathbb R$, let
$n_\theta:=(\cos\theta,\sin\theta)$ and $\tau_\theta:=(-\sin\theta,\cos\theta)$.
For any planar convex body $K$ and $\theta\in\mathbb{R}$, let
$\widetilde h_K(\theta):=h_K(n_\theta).$
Then the following assertions hold.

\begin{itemize}
\item[$\mathrm{(i)}$] Let $K$ be a $C^k_+$ planar convex body. Then
$\widetilde h_K\in C^k(\mathbb R)$. Moreover, for any given
$\theta\in\mathbb R$, there exists a unique point
$x_K(\theta)\in\partial K$ whose outward unit normal is $n_\theta$. Furthermore,
the map $x_K:\mathbb R\to\partial K$ is $C^{k-1}$ and, for any $\theta\in\mathbb{R},$
\begin{align*}
x_K(\theta)=\widetilde h_K(\theta)n_\theta
+\widetilde h_K'(\theta)\tau_\theta\ \ \text{and}\ \ x_K'(\theta)
=\rho_K(\theta)\tau_\theta,
  \end{align*}
where $\rho_K(\theta):=\widetilde h_K(\theta)+\widetilde h_K''(\theta)>0.$
Consequently,
\begin{align*}
\operatorname{Per}(K)=\int_0^{2\pi}\rho_K(\theta)\,d\theta
=\int_0^{2\pi}\widetilde h_K(\theta)\,d\theta .
\end{align*}

\item[$\mathrm{(ii)}$] Conversely, let $h\in C^k(\mathbb R)$ be
$2\pi$-periodic and assume that, for any $\theta\in\mathbb{R},$
\begin{align*}
\rho(\theta):=h(\theta)+h''(\theta)>0.
\end{align*}
For any $\theta\in\mathbb{R},$ define
\begin{align*}
x_h(\theta):=h(\theta)n_\theta+h'(\theta)\tau_\theta.
\end{align*}
Then $x_h$ is injective on $\mathbb{T}_{2\pi}$ and there
exists a $C^k_+$ planar convex body $K_h$ such that
$\partial K_h=x_h(\mathbb R)$
and, for any $\theta\in\mathbb R$, $h_{K_h}(n_\theta)=h(\theta).$
\end{itemize}
\end{lemma}

Now, we prove Proposition \ref{prop:lim}.
\begin{proof}[Proof of Proposition \ref{prop:lim}]
Let $\psi\in C^2(\mathbb R)$ be $2\pi$-periodic and choose
$\varepsilon_\psi\in(0,\infty)$ such that, for any
$\varepsilon\in(-\varepsilon_\psi,\varepsilon_\psi)$ and  $\theta\in\mathbb R$,
\begin{align}\label{eq:rhoepszheng}
\frac12\le1+\varepsilon\left[\psi(\theta)+\psi''(\theta)\right]\le2.
\end{align}
For any $\varepsilon\in (-\varepsilon_\psi,\varepsilon_\psi)$ and $\theta\in\mathbb{R}$, define
\begin{align*}
h_\varepsilon(\theta):=1+\varepsilon\psi(\theta)\ \ \text{and}\ \
\rho_\varepsilon(\theta):=h_\varepsilon(\theta)+h_\varepsilon''(\theta)
=1+\varepsilon\left[\psi(\theta)+\psi''(\theta)\right].
\end{align*}
By Lemma \ref{lem:nangle}, we conclude that, for any
given $\varepsilon\in (-\varepsilon_\psi,\varepsilon_\psi)$, there exists
a $C^2_+$ planar convex body $K_{\varepsilon,\psi}$
whose support function is $h_{\varepsilon}.$
Moreover, if $x_\varepsilon(\theta)$ denotes the boundary point with outward unit
normal $n_\theta$, then
\begin{align}\label{eq:xedao}
x_\varepsilon(\theta)=h_\varepsilon(\theta)n_\theta+h_\varepsilon'(\theta)\tau_\theta
\ \ \text{and}\ \ x_\varepsilon'(\theta)=\rho_\varepsilon(\theta)\tau_\theta,
\end{align}
In particular, when $\varepsilon:=0$, then $K_{\varepsilon,\psi}=\overline{B}_1.$
For any $\theta,\eta\in \mathbb{R},$ let
\begin{align*}
r_\varepsilon(\theta,\eta):=x_\varepsilon(\eta)-x_\varepsilon(\theta)
\ \ \text{and}\ \ \ell_\varepsilon(\theta,\eta):=|r_\varepsilon(\theta,\eta)|.
\end{align*}
Moreover, let
\begin{align*}
a_\varepsilon(\theta,\eta):=-r_\varepsilon(\theta,\eta)\cdot n_\theta
\ \ \text{and}\ \
b_\varepsilon(\theta,\eta):=r_\varepsilon(\theta,\eta)\cdot n_\eta.
\end{align*}
Then, from the support line property \eqref{eq:x01}, we deduce that, for any $\theta,
\eta\in (0,2\pi)$ with $\theta\not=\eta,$
$a_\varepsilon(\theta,\eta)\geq 0$ and $b_\varepsilon(\theta,\eta)\geq 0$.

Now, we show that, for any $\varepsilon\in (-\varepsilon_\psi,\varepsilon_\psi)$,
\begin{align}\label{eq:bianjiedui}
C_q(K_{\varepsilon,\psi})=\int_0^{2\pi}\int_0^{2\pi}
\ell_\varepsilon(\theta,\eta)^{q-3} a_\varepsilon(\theta,\eta)b_\varepsilon(\theta,\eta)
\rho_\varepsilon(\theta)\rho_\varepsilon(\eta)\,d\theta\,d\eta.
\end{align}
For any $\beta\in(0,2\pi)$ and $\zeta\in\mathbb R$, let
$L_{\beta,\zeta}:=\zeta \tau_\beta+\mathbb R n_\beta$ and $\ell_{K_{\varepsilon,\psi}}(\beta,\zeta)
:=\mathcal H^1(K_{\varepsilon,\psi}\cap L_{\beta,\zeta}).$
Define $p_c:=x_\varepsilon(0)=x_\varepsilon(2\pi)$ and
\begin{align*}
V_\varepsilon:=\left\{(\beta,\zeta)\in(0,2\pi)\times\mathbb R:
L_{\beta,\zeta}\cap\operatorname{int}K_{\varepsilon,\psi}\ne\varnothing,
\ p_c\notin L_{\beta,\zeta}\right\}.
\end{align*}
Then
\begin{align*}
C_q(K_{\varepsilon,\psi})=\int_{V_\varepsilon}\ell_{K_{\varepsilon,\psi}}
(\beta,\zeta)^q\,d\zeta\,d\beta.
\end{align*}
Let $U_\varepsilon$ be the set of all $(\theta,\eta)\in(0,2\pi)^2$ such
that $\theta\ne\eta$ and the direction
$r_\varepsilon(\theta,\eta)/\ell_\varepsilon(\theta,\eta)$ is not $n_0$.
For any $(\theta,\eta)\in U_\varepsilon$, let
$\beta(\theta,\eta)\in(0,2\pi)$
be  such that $\frac{r_\varepsilon(\theta,\eta)}{\ell_\varepsilon(\theta,\eta)}=
n_{\beta(\theta,\eta)}.$ Define
\begin{align*}
\zeta(\theta,\eta):=
x_\varepsilon(\theta)\cdot \tau_{\beta(\theta,\eta)}\ \ \text{and}\ \ T_\varepsilon(\theta,\eta)
:=(\beta(\theta,\eta),\zeta(\theta,\eta)).
\end{align*}
Obviously, $T_\varepsilon:U_\varepsilon\to V_\varepsilon$ is a $C^1$ bijection.

We next compute the Jacobian of $T_\varepsilon$.  Instead of
differentiating $\beta(\theta,\eta)$ and $\zeta(\theta,\eta)$ directly, we
describe the graph of $T_\varepsilon$ by two scalar equations. Let
\begin{align*}
\operatorname{graph}(T_\varepsilon):=\left\{
(\theta,\eta,\beta,\zeta)\in U_\varepsilon\times V_\varepsilon:
(\beta,\zeta)=T_\varepsilon(\theta,\eta)\right\}.
\end{align*}
For any $(\theta,\eta,\beta,\zeta)\in \operatorname{graph}(T_\varepsilon),$ let
\begin{align*}
F_1(\theta,\eta,\beta,\zeta):=\left[x_\varepsilon(\eta)-x_\varepsilon(\theta)\right]\cdot \tau_\beta
\ \ \text{and}\ \
F_2(\theta,\eta,\beta,\zeta):=x_\varepsilon(\theta)\cdot \tau_\beta-\zeta.
\end{align*}
The equation $F_1=0$ means that the chord vector
$x_\varepsilon(\eta)-x_\varepsilon(\theta)$ is parallel to $n_\beta$,
whereas $F_2=0$ means that the point $x_\varepsilon(\theta)$ lies on the
line $L_{\beta,\zeta}$.
Thus, on the graph of $T_\varepsilon$, we
have $F_1=F_2\equiv0.$
This implies that
\begin{align}\label{eq:yakebi}
\left|\det \frac{\partial(\beta,\zeta)}{\partial(\theta,\eta)}\right|
=\frac{\left|\det \frac{\partial(F_1,F_2)}{\partial(\theta,\eta)}\right|}
{\left|\det \frac{\partial(F_1,F_2)}{\partial(\beta,\zeta)}\right|}.
\end{align}
By a direct computation, we obtain
\begin{align*}
\frac{\partial F_1}{\partial\beta}=
\left[x_\varepsilon(\eta)-x_\varepsilon(\theta)\right]\cdot\partial_\beta\tau_\beta
=\left[x_\varepsilon(\eta)-x_\varepsilon(\theta)\right]\cdot (-n_{\beta})=
-\ell_\varepsilon(\theta,\eta),
\end{align*}
$\frac{\partial F_1}{\partial\zeta}=0$,
$\frac{\partial F_2}{\partial\beta}=x_\varepsilon(\theta)\cdot (-n_{\beta}),$
and $\frac{\partial F_2}{\partial\zeta}=-1.$ Thus,
\begin{align}\label{eq:yakebi2}
\left|\det\frac{\partial(F_1,F_2)}{\partial(\beta,\zeta)}\right|=\ell_\varepsilon.
\end{align}
Moreover, $\frac{\partial F_1}{\partial\theta}=
-x_\varepsilon'(\theta)\cdot\tau_\beta$,
$\frac{\partial F_1}{\partial\eta}=x_\varepsilon'(\eta)\cdot\tau_\beta$,
$\frac{\partial F_2}{\partial\theta}= x_\varepsilon'(\theta)\cdot\tau_\beta,$
and $\frac{\partial F_2}{\partial\eta}=0.$ Thus,
\begin{align}\label{eq:yakebi3}
\left|\det\frac{\partial(F_1,F_2)}{\partial(\theta,\eta)}\right|
=|x_\varepsilon'(\theta)\cdot \tau_\beta||x_\varepsilon'(\eta)\cdot \tau_\beta|.
\end{align}
From the facts that $r_\varepsilon(\theta,\eta)=\ell_\varepsilon(\theta,\eta)n_\beta$
on $\operatorname{graph}(T_\varepsilon)$ and  $a_\varepsilon(\theta,\eta)\geq0$, we infer that
\begin{align*}
a_\varepsilon(\theta,\eta)=-r_\varepsilon(\theta,\eta)\cdot n_\theta
=-\ell_\varepsilon(\theta,\eta)n_\beta\cdot n_\theta
=\ell_\varepsilon(\theta,\eta)|n_\theta\cdot n_\beta|.
\end{align*}
This implies that
\begin{align}\label{eq:deng}
|x_\varepsilon'(\theta)\cdot \tau_\beta|=\rho_\varepsilon(\theta)|\tau_\theta\cdot\tau_\beta|
=\rho_\varepsilon(\theta)|n_\theta\cdot n_\beta|
=\rho_\varepsilon(\theta)\frac{a_\varepsilon(\theta,\eta)}
{\ell_\varepsilon(\theta,\eta)}.
\end{align}
Similarly, we also have
\begin{align*}
|x_\varepsilon'(\eta)\cdot \tau_\beta|=\rho_{\varepsilon}(\eta)|\tau_\eta\cdot \tau_\beta|
=\rho_{\varepsilon}(\eta)|n_\eta\cdot n_\beta|=
\rho_\varepsilon(\eta)\frac{b_\varepsilon(\theta,\eta)}
{\ell_\varepsilon(\theta,\eta)}.
\end{align*}
Using this, \eqref{eq:deng}, \eqref{eq:yakebi}, \eqref{eq:yakebi2}, and \eqref{eq:yakebi3},
we obtain
\begin{align*}
\left|\det \frac{\partial(\beta,\zeta)}{\partial(\theta,\eta)}\right|
=\frac{|x_\varepsilon'(\theta)\cdot \tau_\beta||x_\varepsilon'(\eta)\cdot \tau_\beta|}
{\ell_\varepsilon}
=\frac{a_\varepsilon(\theta,\eta)b_\varepsilon(\theta,\eta)
\rho_\varepsilon(\theta)\rho_\varepsilon(\eta)}
{\ell_\varepsilon(\theta,\eta)^3}>0.
\end{align*}
From this, the fact that $T_\varepsilon:U_\varepsilon\to V_\varepsilon$ is a $C^1$
bijection, and the inverse function theorem, it follows that
$T_\varepsilon:U_\varepsilon\to V_\varepsilon$ is a $C^1$
diffeomorphism. This, combined with  the change-of-variables formula, further implies
that \eqref{eq:bianjiedui} holds.

Next, we prove that $\varepsilon\mapsto C_q(K_{\varepsilon,\psi})$
is $C^2$ for $\varepsilon\in(-\varepsilon_\psi,\varepsilon_\psi).$
For any $\theta,\eta\in(0,2\pi)$ with $\theta\not=\eta$, let
\begin{align*}
P_{\varepsilon}(\theta,\eta):=a_\varepsilon(\theta,\eta)b_\varepsilon(\theta,\eta)
\rho_\varepsilon(\theta)\rho_\varepsilon(\eta)
\ \ \text{and}\ \
F_\varepsilon(\theta,\eta):=\ell_\varepsilon(\theta,\eta)^{q-3}
P_{\varepsilon}(\theta,\eta).
\end{align*}
Let $\varepsilon\in(-\varepsilon_\psi,\varepsilon_\psi).$
By \eqref{eq:bianjiedui} and the dominated convergence theorem,
we find that, to obtain the desired   conclusion, it suffices
to show that, for any  $\theta,\eta\in(0,2\pi)$ with $\theta\not=\eta$,
\begin{align}\label{eq:Fepskongzhi}
\sum_{j=0}^2|\partial_\varepsilon^jF_\varepsilon(\theta,\eta)|\lesssim
|\delta(\theta,\eta)|^{q+1}\mathbf 1_{\{|\delta(\theta,\eta)|\le1\}}
+\mathbf 1_{\{|\delta(\theta,\eta)|>1\}}.
\end{align}
Let $\theta,\eta\in(0,2\pi)$ with $\theta\not=\eta$.  Choose the
representative of $\eta-\theta$ modulo $2\pi$ lying in $(-\pi,\pi]$ and
denote it by $\delta=\delta(\theta,\eta)$. Then
\begin{align*}
\eta-\theta-\delta\in2\pi\mathbb Z\ \ \text{and}\ \ |\delta|
=\min_{k\in\mathbb Z}|\theta-\eta+2\pi k|.
\end{align*}
From \eqref{eq:xedao}, it follows that
\begin{align}\label{eq:repsjifen}
r_\varepsilon(\theta,\eta)
=x_\varepsilon(\theta+\delta)-x_\varepsilon(\theta)
=\int_0^\delta x'_{\varepsilon}(\theta+t)\,dt
=\int_0^\delta\rho_\varepsilon(\theta+t)\tau_{\theta+t}\,dt,
\end{align}
\begin{align}\label{eq:aepsjifen}
a_\varepsilon(\theta,\eta)=-\int_0^\delta
\rho_\varepsilon(\theta+t)\tau_{\theta+t}\cdot n_\theta\,dt
=\int_0^\delta\rho_\varepsilon(\theta+t)\sin t\,dt,
\end{align}
and
\begin{align}\label{eq:bepsjifen}
b_\varepsilon(\theta,\eta)=\int_0^\delta\rho_\varepsilon(\theta+t)\tau_{\theta+t}\cdot n_\eta\,dt
=\int_0^{\delta}\rho_{\varepsilon}(\theta+\delta-t)\sin t\,dt.
\end{align}
We first consider the case $0<|\delta(\theta,\eta)|\le 1.$
By \eqref{eq:repsjifen} and \eqref{eq:rhoepszheng}, we conclude that
\begin{align}\label{eq:ellshanglocal}
\ell_\varepsilon(\theta,\eta)=|r_\varepsilon(\theta,\eta)|
\leq\int_0^{|\delta(\theta,\eta)|} \rho_{\varepsilon}(\theta+t)\,dt\leq
2|\delta(\theta,\eta)|
\end{align}
and
\begin{align}\label{eq:ellxialocal}
\ell_\varepsilon(\theta,\eta)\geq
|r_\varepsilon(\theta,\eta)\cdot \tau_\theta|
&=\left|\int_0^{\delta(\theta,\eta)} \rho_{\varepsilon}(\theta+t)\cos t\,dt\right|\notag\\
&\geq\int_0^{|\delta(\theta,\eta)|}\frac{\cos t}{2}\,dt\geq
\frac{\cos 1 }{2}|\delta(\theta,\eta)|.
\end{align}
Moreover,
\begin{align*}
a_\varepsilon(\theta,\eta)\sim\int_0^{|\delta(\theta,\eta)|} \sin t\,dt
=1-\cos |\delta(\theta,\eta)|
\sim|\delta(\theta,\eta)|^2
\end{align*}
and $b_\varepsilon(\theta,\eta)\sim |\delta|^2.$
Using this, \eqref{eq:ellxialocal}, and \eqref{eq:ellshanglocal}, we find that
\begin{align}\label{eq:jibenguji}
\ell_\varepsilon(\theta,\eta)\sim |\delta(\theta,\eta)|\ \ \text{and}\ \ a_\varepsilon(\theta,\eta),
b_\varepsilon(\theta,\eta)\sim |\delta(\theta,\eta)|^2.
\end{align}
From the definition of $\rho_{\varepsilon}$ and \eqref{eq:rhoepszheng}, we deduce that
\begin{align*}
\sum_{j=0}^1|\partial_\varepsilon^j\rho_\varepsilon(\theta)|+
|\partial_\varepsilon^j\rho_\varepsilon(\eta)|\lesssim
1\ \ \text{and}\ \ \partial_\varepsilon^2\rho_\varepsilon(\theta)=\partial_\varepsilon^2
\rho_\varepsilon(\eta)=0.
\end{align*}
By this, \eqref{eq:aepsjifen}, and \eqref{eq:bepsjifen}, we conclude that
\begin{align*}
\sum_{j=0}^1|\partial_\varepsilon^j a_\varepsilon(\theta,\eta)|+
|\partial_\varepsilon^j b_\varepsilon(\theta,\eta)|\lesssim
|\delta(\theta,\eta)|^2
\ \ \text{and}\ \ \partial^2_{\varepsilon}a_\epsilon(\theta,\eta)
=\partial^2_{\varepsilon}b_\epsilon(\theta,\eta)=0.
\end{align*}
This implies that
\begin{align}\label{eq:dijiu}
\sum_{j=0}^2|\partial_\varepsilon^j P_\varepsilon(\theta,\eta)|
\lesssim|\delta(\theta,\eta)|^4.
\end{align}
On the other hand, from \eqref{eq:repsjifen} and \eqref{eq:jibenguji}, we infer that
\begin{align*}
\sum_{j=0}^1|\partial_\varepsilon^j r_\varepsilon(\theta,\eta)|\lesssim
|\delta(\theta,\eta)|,
\end{align*}
\begin{align*}
|\partial_\varepsilon\ell_\varepsilon(\theta,\eta)|=
\left|\frac{r_\varepsilon(\theta,\eta)\cdot\partial_\varepsilon r_\varepsilon(\theta,\eta)}
{\ell_\varepsilon(\theta,\eta)}\right|\lesssim
|\partial_\varepsilon r_\varepsilon(\theta,\eta)|\lesssim
|\delta(\theta,\eta)|,
\end{align*}
and
\begin{align*}
|\partial_\varepsilon^2\ell_\varepsilon(\theta,\eta)|
=\left|\frac{|\partial_\varepsilon r_\varepsilon(\theta,\eta)|^2}
{\ell_\varepsilon(\theta,\eta)}-
\frac{(r_\varepsilon(\theta,\eta)\cdot\partial_\varepsilon r_\varepsilon(\theta,\eta))^2}
{\ell_\varepsilon^3(\theta,\eta)}\right|\lesssim
|\delta(\theta,\eta)|.
\end{align*}
This, together with \eqref{eq:dijiu}, implies that, for any $\theta,\eta\in (0,2\pi)$
with $0<|\delta(\theta,\eta)|\leq 1,$
\begin{align}\label{eq:wushidu}
\sum_{j=0}^2 |\partial_\varepsilon^jF_\varepsilon(\theta,\eta)|\lesssim
|\delta(\theta,\eta)|^{q+1}.
\end{align}
Assume that $\theta,\eta\in (0,2\pi)$ with $|\delta(\theta,\eta)|\geq 1.$
Recall that, for any $\beta\in\mathbb{R},$
\begin{align*}
x_\varepsilon(\beta)= n_\beta+\varepsilon\psi(\beta)n_\beta
+\varepsilon\psi'(\beta)\tau_\beta.
\end{align*}
Thus,
\begin{align*}
\ell_0(\theta,\eta)&=|x_0(\eta)-x_0(\theta)|
=|n_\eta-n_\theta|=|n_{\theta+\delta(\theta,\eta)}-n_\theta|\\
&=2\left|\sin\frac{\delta(\theta,\eta)}{2}\right|\geq
2\sin\frac12=:c_0.
\end{align*}
Moreover, we have
\begin{align*}
|x_\varepsilon(\beta)-x_0(\beta)|\le|\varepsilon|\left[
\|\psi\|_{L^\infty(\mathbb{R})}+\|\psi'\|_{L^\infty(\mathbb{R})}\right]
\end{align*}
and
\begin{align*}
|\ell_\varepsilon(\theta,\eta)-\ell_0(\theta,\eta)|\leq
|x_\varepsilon(\eta)-x_0(\eta)|+|x_\varepsilon(\theta)-x_0(\theta)|
\le2|\varepsilon|\left[\|\psi\|_{L^\infty(\mathbb{R})}+\|\psi'\|_{L^\infty(\mathbb{R})}\right].
\end{align*}
Thus, $\ell_\varepsilon\to\ell_0$ uniformly in $(\theta,\eta)$ as
$\varepsilon\to0$. This implies that, after decreasing $\varepsilon_\psi$ if necessary, we obtain
\begin{align}\label{eq:ellyuanxia}
\ell_\varepsilon(\theta,\eta)\ge\ell_0(\theta,\eta)
-|\ell_\varepsilon(\theta,\eta)-\ell_0(\theta,\eta)|\ge\frac{c_0}{2}.
\end{align}
On the other hand, by \eqref{eq:repsjifen}, \eqref{eq:aepsjifen}, \eqref{eq:bepsjifen},
and the fact  that $|\delta(\theta,\eta)|\leq \pi,$ it is easy to check that
\begin{align*}
\sum_{j=0}^2 |\partial_\varepsilon^j\ell_\varepsilon(\theta,\eta)|
\lesssim1 \quad\text{and}\quad
\sum_{j=0}^2 |\partial_\varepsilon^jP_\varepsilon(\theta,\eta)|\lesssim1.
\end{align*}
From this and \eqref{eq:ellyuanxia}, it follows that, for any $\theta,\eta\in (0,2\pi)$
with $|\delta(\theta,\eta)|\geq 1,$
$\sum_{j=0}^2|\partial_\varepsilon^jF_\varepsilon(\theta,\eta)|\lesssim 1,$
which, together with \eqref{eq:wushidu}, implies that \eqref{eq:Fepskongzhi} holds.

Let
\begin{align*}
L_q(\psi):=\left.\frac{d}{d\varepsilon}\right|_{\varepsilon=0}C_q(K_{\varepsilon,\psi}).
\end{align*}
Since we have proved that $\varepsilon\mapsto C_q(K_{\varepsilon,\psi})$
is $C^2$ for $\varepsilon\in(-\varepsilon_\psi,\varepsilon_\psi),$  it follows that
\begin{align}\label{eq:buhaip}
C_q(K_{\varepsilon,\psi})=C_q(B_1)+\varepsilon L_q(\psi)+
O(\varepsilon^2)\ \ \text{as}\ \varepsilon\to0.
\end{align}
Let $C^2_{\mathrm{per}}(0,2\pi)$ denote the real vector space of all $2\pi$-periodic
functions in $C^2(\mathbb R)$. Then, by  \eqref{eq:bianjiedui} and the pointwise estimate
\eqref{eq:Fepskongzhi}, we conclude that
\begin{align*}
|L_q(\psi)|\lesssim\max_{0\le m\le2}\|\psi^{(m)}\|_{L^\infty(\mathbb{R})}.
\end{align*}
Thus, $L_q$ is a  linear bounded functional on $C^2_{\mathrm{per}}(0,2\pi)$.
Let $a\in\mathbb{R}$ and $R_a\psi(\theta):=\psi(\theta+a)$ for any $\theta\in\mathbb{R}.$
Observe that the convex body with support function $1+\varepsilon R_a\psi$
is a rotation of the convex body with support function $1+\varepsilon\psi$
and  $C_q$ is invariant under rotations. This implies that, for any $a\in\mathbb{R},$
 $L_q(\psi)=L_q(R_a\psi).$
Using this and the continuity and linearity of $L_q$, we obtain
\begin{align*}
L_q(\psi)&=\frac{1}{2\pi}\int_0^{2\pi}L_q(R_a\psi)\,da
=L_q\left(\frac1{2\pi}\int_0^{2\pi}R_a\psi\,da\right)
=L_q\left(\frac1{2\pi}\int_0^{2\pi}\psi(a)\,da\right)\\
&=\frac{L_q(1)}{2\pi}\int_0^{2\pi}\psi(a)\,da.
\end{align*}
From this and \eqref{eq:buhaip}, we deduce that
\begin{align}\label{eq:zhangkai}
C_q(K_{\varepsilon,\psi})=C_q(B_1)+\varepsilon \frac{L_q(1)}{2\pi}\int_0^{2\pi}\psi(a)\,da
+O(\varepsilon^2)\ \ \text{as}\ \varepsilon\to0.
\end{align}

For any $\theta\in\mathbb{R},$ let
$\varphi(\theta):=\widetilde h_K(\theta)-1.$
Then, by Lemma \ref{lem:nangle}(i), we have
\begin{align}\label{eq:phipingjun}
\int_0^{2\pi}\varphi(\theta)\,d\theta=0.
\end{align}
Using Lemma \ref{lem:suppline}(ii), we obtain,  for any $r\in (0,\infty)$ and $\theta\in\mathbb R$,
\begin{align*}
\widetilde h_{\frac{K+rB_1}{1+r}}(\theta)=\frac{\widetilde h_K(\theta)+r}{1+r}
=1+\frac{\varphi(\theta)}{1+r}.
\end{align*}
From this  and  the uniqueness of a convex body with a given
support function (see Lemma \ref{lem:nangle}), we infer that, for all sufficiently
large $r$, the planar convex body $\frac{K+r\overline{B}_1}{r+1}$ is precisely
$K_{\frac1{1+r},\varphi}$. Applying \eqref{eq:zhangkai} with $\psi:=\varphi$
and using \eqref{eq:phipingjun}, we find that
\begin{align*}
C_q\left(\frac{K+rB_1}{1+r}\right)-C_q(B_1)=O\left((1+r)^{-2}\right)
\qquad\text{as }r\to\infty.
\end{align*}
This, together with \eqref{eq:shensuo}, implies that
\begin{align*}
C_q(K+rB_1)-C_q(B_{r+1})&=(1+r)^{q+1}\left[C_q\left(\frac{K+rB_1}{1+r}\right)-C_q(B_1)
\right]  \\
&=(1+r)^{q+1}O((1+r)^{-2})=O((1+r)^{q-1})=O(r^{q-1})
\end{align*}
as $r\to\infty$. This shows \eqref{eq:lim}, which completes the proof of Proposition \ref{prop:lim}.
\end{proof}

\subsection{Proof of the  Sharp Convex Chord Inequality}\label{sec:proof}
In this subsection, we prove Theorem \ref{thm:conv}.
Let $K\subset\mathbb{R}^2$ be a planar convex body. For any $v\in\mathbb{S}^1$, let
\begin{align*}
w_K(v):=h_K(v)+h_K(-v).
\end{align*}
Then $w_K(v)$ is the distance between the two parallel lines
$\ell_K(v)$ and $\ell_K(-v)$.
The following identity is the planar case of Cauchy's surface-area formula
(see, for instance, \cite[(5.73)]{s14}).
\begin{lemma}\label{lem:cauchy}
Let $K$ be a planar convex body. Then
\begin{align*}
\operatorname{Per}(K)=\frac12\int_{\mathbb{S}^1} w_K(v)\,d\sigma(v).
\end{align*}
\end{lemma}

The following conclusion follows
directly from \cite[Theorem 3.4.1]{s14} and the remark following its proof.

\begin{lemma}\label{lem:smooth}
Let $K$ be a planar convex body. Then there exists a sequence
$\{K_j\}_{j\in\mathbb{N}}$ of $C^\infty_+$ planar convex bodies  such that $K_j\to K$ in the Hausdorff distance as $j\to\infty.$
\end{lemma}
The following result is well known. For lack of a suitable reference, we include a proof for completeness.
\begin{lemma}\label{lem:cont}
Let $\{K_j\}_{j\in\mathbb{N}}$ and $K$ be planar convex bodies and $q\in (0,1).$
Assume that $K_j\to K$ in the Hausdorff distance as $j\to\infty$. Then
\begin{align}\label{eq:wusuowei}
\lim_{j\to\infty}\operatorname{Per}(K_j)=\operatorname{Per}(K)
\end{align}
and
\begin{align}\label{eq:wusuowei2}
\lim_{j\to\infty}C_q(K_j)=C_q(K).
\end{align}
\end{lemma}

\begin{proof}
By Lemma \ref{lem:suppline}(iii), we conclude that
$w_{K_j}\to w_K$ uniformly on $\mathbb{S}^1$ as $j\to\infty.$
From this and Lemma \ref{lem:cauchy}, we deduce that \eqref{eq:wusuowei} holds.

It remains to show \eqref{eq:wusuowei2}.
Let $E\subset \mathbb{R}^2$ be a planar convex body.
For any given $u\in\mathbb{S}^1$ and $z\in u^{\perp}$, let
\begin{align*}
m_E(u,z):=\mathcal{H}^1\left(E\cap\ell(u,z)\right).
\end{align*}
Since the section of a convex body by a line is either empty or a compact
interval,  it follows that
\begin{align*}
C_q(E)=\int_{\mathbb{S}^1}\int_{u^\perp}m_E(u,z)^q\,dz\,d\sigma(u).
\end{align*}
Then, by the dominated convergence theorem,  to prove \eqref{eq:wusuowei2}, we only need to show that,
for  almost  every $u\in\mathbb{S}^1$ and $z\in u^{\perp},$
\begin{align}\label{eq:mKjdaomK}
\lim_{j\to\infty}m_{K_j}(u,z)= m_K(u,z).
\end{align}

We first prove that, for almost every $x\in\mathbb{R}^2,$
\begin{align}\label{eq:qianbaidu}
\lim_{j\to\infty}\mathbf{1}_{K_j}(x)=\mathbf{1}_{K}(x).
\end{align}
Let  $x\in\mathbb{R}^2$. If $x\notin K$.  Then, from the  fact that  $K$ is compact, we infer that there
exists $\delta\in (0,\infty)$ such that $\operatorname{dist}(x,K)>\delta$.
Since $\lim_{j\to\infty}d_{H}(K_j,K)=0,$ it follows that
there exists $j_0\in\mathbb{N}$ such that, for any $j\in\mathbb{N}\cap [j_0,\infty)$,
$K_j\subset K+\delta B_1$. This implies  that $\mathbf{1}_{K_j}(x)=\mathbf{1}_{K}(x)=0$ for any $j\in\mathbb{N}\cap [j_0,\infty)$.
Thus, \eqref{eq:qianbaidu} holds in this case. Now, let $x\in \operatorname{int}(K).$
By \eqref{eq:yuanchuang}, we find that there exists $\epsilon\in(0,\infty)$ such that, for any $v\in \mathbb{S}^1,$
$x\cdot v\le h_K(v)-\epsilon.$
From the fact that $h_{K_j}\to h_K$ uniformly on $S^1$ as $j\to\infty$, we deduce that, for any $v\in
\mathbb{S}^1$ and for all large $j$,
$x\cdot v\le h_{K_j}(v)$ and hence $x\in K_j$.
This implies that $\mathbf{1}_{K_j}(x)=\mathbf{1}_{K}(x)=1$, which competes the proof of \eqref{eq:qianbaidu}.

Using \eqref{eq:qianbaidu}, we
conclude  that, for any $u\in\mathbb{S}^1$ and for  almost every $t\in(0,\infty)$ and $z\in u^\perp,$
\begin{align*}
\lim_{j\to\infty}\mathbf{1}_{K_j}(z+tu)=\mathbf{1}_K(z+tu),
\end{align*}
which, combined with  the dominated convergence
theorem, further implies that
\begin{align*}
\lim_{j\to\infty}m_{K_j}(u,z)=\lim_{j\to\infty}\int_{\mathbb{R}}\mathbf{1}_{K_j}(z+tu)\,dt
=\int_{\mathbb{R}}\mathbf{1}_K(z+tu)\,dt=m_K(u,z).
\end{align*}
This shows \eqref{eq:mKjdaomK}, which  completes the proof of Lemma \ref{lem:cont}.
\end{proof}

Now, we show Theorem \ref{thm:conv}.
\begin{proof}[Proof of Theorem \ref{thm:conv}]
Using Lemmas \ref{lem:smooth} and \ref{lem:cont}, we conclude that, to prove the present theorem,
we just need to show \eqref{eq:conv} for  $C^\infty_+$ convex body $K$.
By \eqref{eq:shensuo}, without loss of generality,
we may assume
$\operatorname{Per}(K)=2\pi.$
For any $r\in [0,\infty)$, let $K_r:=K+rB_1.$
Using Lemma \ref{lem:suppline}, we find that, for any $r\in [0,\infty)$,
\begin{align*}
\widetilde{h}_{K_r}=\widetilde{h}_{K}+r\widetilde{h}_{B_1}=\widetilde{h}_{K}+r.
\end{align*}
From this, Lemma \ref{lem:nangle}, and the assumption that $K$ is a $C^\infty_{+}$ convex body,
we infer that, for any $r\in [0,\infty)$, $K_r$
is a $C^\infty_{+}$ convex body.
By Lemma \ref{lem:cauchy}, we have
\begin{align}\label{eq:ying}
\operatorname{Per}(K_r)=\int_0^{2\pi}(h_K(\theta)+r)\,d\theta
=\operatorname{Per}(K)+2\pi r=2\pi(r+1)=\operatorname{Per}(B_{r+1}).
\end{align}
For any  $r\in [0,\infty),$ let $H(r):=C_q(K_r)-C_q(B_{r+1})$.
Using  Lemma \ref{lem:cont},
we conclude that $H$ is continuous.
By Proposition \ref{prop:dcomp}
and \eqref{eq:ying}, we find that, for any $r\in [0,\infty),$
\begin{align*}
\left.\frac{d}{ds}\right|_{s=0} C_q(K_r+sB_1)\ge
\left.\frac{d}{ds}\right|_{s=0} C_q(B_{r+1}+sB_1)
\end{align*}
and hence $H'_+(r)\geq 0.$
 This implies that
$H$ is nondecreasing on $[0,\infty)$.
Thus, for any $r\in [0,\infty),$
\begin{align*}
C_q(K)\leq C_q(K_r)-C_q(B_{r+1})+C_q(B_1).
\end{align*}
Letting $r\to\infty$ and using Proposition \ref{prop:lim}, we obtain
$C_q(K)\leq C_q(B_1).$
This finishes the proof of Theorem \ref{thm:conv}.
\end{proof}

\smallskip

\noindent\textbf{Acknowledgements}\quad 
The main mathematical ideas and approaches in this article
originated from the authors. Generative AI tools were used as
supplementary aids during the exploratory and manuscript-preparation
stages, particularly in refining the presentation and examining
technical details. All mathematical arguments and proofs in the
final manuscript were independently verified and finalized by
the authors, who take full responsibility for their correctness.

The authors would also like to thank
Professor Rupert L. Frank for bringing to their attention
his preprint with Professor Paata Ivanisvili, arXiv:2609.14513 [math.AP],
in which Frank and Ivanisvili, independently and simultaneously,
obtained a closely related result. We became aware of their work
only after the present manuscript had already been submitted to the arXiv.
The two approaches are substantially different.
While both arguments involve a reduction to the convex setting,
our proof is based on the parallel flow, whereas
Frank and Ivanisvili use the Euler--Lagrange equation
and additionally classify the equality cases. We should
also mention that the first version of our article was
submitted to arXiv on July 21, 2026, and then this article
was on hold by arXiv till September 17, 2026.
During the period when it was on hold by arXiv,
on August 13 we submitted a revised version to arXiv by adding
the above acknowledgement to AI.

Finally, we would also like to thank Professor E. D\"ohrer for bringing his joint
arXiv preprint \cite{dd26} with Professor A. Dohmen to our attention.
We were unaware of this work when the first version of this manuscript was completed.
We also thank Professor E. D\"ohrer for pointing out the related works
\cite{bgss25,gs24}.

%
%
%

\bigskip

\noindent
Xiaosheng Lin

\smallskip

\noindent
School of  Science, Jimei University,
Xiamen 361021, The People's Republic of China

\smallskip

\noindent {\it E-mail}: \texttt{xslin@jmu.edu.cn}

\bigskip

\noindent Dachun Yang, Wen Yuan and Yangyang Zhang 

\smallskip

\noindent Laboratory of Mathematics and Complex Systems
(Ministry of Education of China),
School of Mathematical Sciences, Institute for Advanced Study,
Beijing Normal University,
Beijing 100875, The People's Republic of China

\smallskip

\noindent{\it E-mails:} \texttt{dcyang@bnu.edu.cn} (D. Yang)

\noindent\phantom{{\it E-mails:}} \texttt{wenyuan@bnu.edu.cn} (W. Yuan)

\noindent\phantom{{\it E-mails:}} \texttt{yangyzhang@bnu.edu.cn} (Y. Zhang)

\bigskip

\noindent Sibei Yang

\medskip

\noindent School of Mathematics and Statistics, Lanzhou University, Lanzhou 730000, The People's Republic of China

\smallskip

\noindent{\it E-mail:} \texttt{yangsb@lzu.edu.cn}

\end{document}